\documentclass[12pt]{amsart}
\usepackage{amssymb}
\usepackage{amsfonts}
\usepackage{amssymb,latexsym}
\usepackage{enumerate}
\usepackage{mathrsfs}
\makeatletter
\@namedef{subjclassname@2010}{%
  \textup{2010} Mathematics Subject Classification}
\makeatother

  [2002/01/19 v2.2g %
    AMS font definitions%
  ]
\DeclareFontFamily{U}{euf}{}
\DeclareFontShape{U}{euf}{m}{n}{%
  <5><6><7><8><9>gen*eufm%
  <10><10.95><12><14.4><17.28><20.74><24.88>eufm10%
  }{}
\DeclareFontShape{U}{euf}{b}{n}{%
  <5><6><7><8><9>gen*eufb%
  <10><10.95><12><14.4><17.28><20.74><24.88>eufb10%
  }{}

  [2002/01/19 v2.2g %
    AMS font definitions%
  ]
\DeclareFontFamily{U}{msb}{}
\DeclareFontShape{U}{msb}{m}{n}{%
  <5><6><7><8><9>gen*msbm%
  <10><10.95><12><14.4><17.28><20.74><24.88>msbm10%
  }{}

  [2002/01/19 v2.2g %
    AMS font definitions%
  ]
\DeclareFontFamily{U}{msa}{}
\DeclareFontShape{U}{msa}{m}{n}{%
  <5><6><7><8><9>gen*msam%
  <10><10.95><12><14.4><17.28><20.74><24.88>msam10%
  }{}

\newtheorem{theorem}{Theorem}[section]
\newtheorem{lemma}[theorem]{Lemma}
\newtheorem{proposition}[theorem]{Proposition}
\newtheorem{corollary}[theorem]{Corollary}

\theoremstyle{definition}
\newtheorem{definition}[theorem]{Definition}

\newtheorem{remark}[theorem]{Remark}

\numberwithin{equation}{section}
\begin{document}

\title[A question of S\'{a}rkozy and S\'{o}s on representation functions] {A question of S\'{a}rkozy and S\'{o}s on representation functions}

\author{Yan Li and Lianrong Ma}

\address{Department of Applied Mathematics, China Agriculture University, Beijing 100083, China}
\email{liyan\_00@mails.tsinghua.edu.cn}
\address{Department of
Mathematics, Tsinghua University, Beijing 100084, China}
\email{lrma@math.tsingua.edu.cn}


\begin{abstract}For $m\geq 1$, let $0<b_0<b_1<\cdots <b_m$ and $\ e_0,e_1,\cdots£¬e_m>0$ be
fixed positive integers. Assume there exists a prime $p$ and an integer
$t>0$ such that $p^t\mid b_0$, but $p^t\nmid b_{i}\ {\rm for}\ 1\leq
i\leq m$.
 Then, we prove that there is no
infinite subset $\mathcal A$ of positive integers, such that the
number of solutions of the following equation
$$n=b_0(a_{0,1}+\cdot +a_{0,e_0})+\cdots +b_m(a_{m,1}+\cdots +a_{m,r_m}),\  a_{i,j}\in \mathcal A$$
is constant for $n$ large enough. This result generalizes the recent
result of Cilleruelo and Ru\'{e} for the bilinear case,
 and answers a question posed by S\'{a}rkozy and S\'{o}s.
\end{abstract}

\maketitle

\def\C{\mathbb C_p}
\def\BZ{\mathbb Z}
\def\Z{\mathbb Z_p}
\def\Q{\mathbb Q_p}
\def\C{\mathbb C_p}
\def\BZ{\mathbb Z}
\def\Z{\mathbb Z_p}
\def\Q{\mathbb Q_p}
\def\psum{\sideset{}{^{(p)}}\sum}
\def\pprod{\sideset{}{^{(p)}}\prod}

\section{Introduction}
Given an infinite subset $\mathcal A$ of positive integers $\mathbb N $, the representation function $r(n,\mathcal A)$ is defined as
$$r(n,\mathcal A)=\#\{(a,a')|n=a+a',\ a,a'\in \mathcal A\}.$$
This function was initially studied by Erd\"{o}s and Tur\'{a}n
~\cite{Erdos}. In that paper, they made the following important conjecture.

\textbf{Conjecture}: (Erd\"{o}s and Tur\'{a}n): If $\mathcal
A\subseteq \mathbb N$ and $r(n,\mathcal A)>0$ for $n>n_0$ (i.e.,
$\mathcal A$ is an asymptotic basis of order 2), then $r(n,\mathcal
A)$ cannot be bounded.

As an evidence, Erd\"{o}s and Tur\'{a}n ~\cite{Erdos} found, by means of
analytic arguments, that $r(n, A)$ cannot be constant for $n$ large
enough.

Dirac ~\cite{Dirac} showed an elementary proof also exists: obviously,
$r(n,\mathcal A)$ is odd when $n=2a,\ a\in \mathcal A$, and even,
otherwise. Moreover, by using the technique of generating functions,
he gave a short and elegant proof that
$$r^{+}(n,\mathcal A)=\#\{(a,a')|n=a+a',\ a,a'\in \mathcal A,\ a\le a'\}$$ cannot be constant, either.

Ruzsa made an surprising example which shows that the above
conjecture does not hold if one replaces $a+a'$ with $a+2a'$.

\textbf{Example of Ruzsa}: Let $$\mathcal A=\{a:a=\sum \limits
_{i=0}^{+\infty}\varepsilon _i2^{2i},\varepsilon _i=0\ {\rm or}\
1\}.$$ Then, for $n\in \mathbb N$, the representation function
$$r_{1,2}(n,\mathcal A)=\#\{(a,a')|n=a+2a',a,a'\in \mathcal A\}$$ is always 1.

Replacing base 2 by base $k$, Ruzsa's example still works, see
earlier arguments of Moser~\cite{Moser}, whose approach is also through
generating functions.

More generally, S\'{a}rkozy and S\'{o}s \cite{Sarkozy} asked the following
question on the representation function of multi-linear forms.

\textbf{Question}: For which $(c_1,\cdots ,c_k)$, can the
representation function
$$r_{c_1,\cdots ,c_k}(n,\mathcal A)=\{(a_1,\cdots ,a_k)|c_1a_1+\cdots +c_ka_k=n,\ a_1,\cdots ,a_k\in \mathcal A\}$$ be constant for $n$ large enough.

Recently, Cilleruelo and Ru\'{e} ~\cite{Cilleruelo} gave a partial answer to
the above question.

\textbf{Theorem} (Cilleruelo and Ru\'{e}): Let $1<c_1<c_2$ and ${\rm
gcd}(c_1,c_2)=1$. There is no infinite subset $\mathcal A$ of
positive integers such that $r_{c_1,c_2} (n, \mathcal A)$ is
constant for $n$ large enough.

Combining the earlier work of Moser ~\cite{Moser}, they completely solved
S\"{a}rkozy and S\'{o}s's question for bilinear forms.

Every multilinear form $c_1x_1+\cdots +c_kx_k$ can be uniquely written as
$$b_0(x_{0,1}+\cdots +x_{0,e_0})+b_1(x_{1,1}+\cdots +x_{1,e_1})+\cdots +b_m(x_{m,1}+\cdots +\cdots +x_{m,e_m})$$ with
 $$0<b_0<b_1<\cdots <b_m\ {\rm and}\ e_0,e_1,\cdots£¬e_m>0.$$ We call
$e_0,e_1,\cdots£¬e_m$ the multiplicities of $b_0,b_1\cdots ,b_m$,
respectively. Denote $${\mathrm M}=\{(b_0,e_0),\cdots ,(b_m,e_m)\}.$$
The representation function of $n$ with respect to ${\mathrm M}$, or equivalently with respect to the multi-linear form $c_1x_1+\cdots +c_kx_k$,
 is
the number of solutions of the equation
$$n=b_0(a_{0,1}+\cdot +a_{0,e_0})+\cdots +b_m(a_{m,1}+\cdots +a_{m,r_m})$$
with $a_{i,j}\in \mathcal A$. We denote this value by $r_{\mathrm
M}(n,\mathcal A)$.

The main result of Ru\'{e}~\cite{Rue} implied that $r_{\mathrm
M}(n,\mathcal A)$ cannot be constant for $n$ large enough if ${\rm
gcd }(e_0,\cdots ,e_m)\ge 2$. The main tool which he used in the proof is
from analytic combinatorics.

\begin{remark}From Ru\'{e}'s~\cite{Rue} result, the case of $m=0$ is clear. That is,
for $m=0$, $r_{\mathrm M}(n,\mathcal A)$ can be constant for $n$
large enough if and only if ${\rm M}=\{(1,1)\}$. So from now on, we will always
assume $m\geq 1$ unless specification.
\end{remark}
In this paper, we will prove the following theorem .

\begin {theorem}\label{th1new} Let $m\ge 1$. Assume there exists a prime $p$ and a positive integer $t$ such that
$p^t\mid b_0,\ {\rm but}\ p^t\nmid b_{i}\ {\rm for}\ 1\leq i\leq m.
$
 Then, for any $\mathcal A\subseteq \mathbb N$,
$r_{\mathrm M}(n,\mathcal A)$ cannot be constant for $n$ large
enough.
\end{theorem}
 Note that the conditions of Theorem \ref{th1new} include the case: $b_0,b_1,\cdots ,b_m$ are pairwise coprime.
Therefore, the above theorem generalizes the theorem of Cilleruelo and
Ru\'{e}~\cite{Cilleruelo} from bilinear forms to multi-linear forms.

It should be noted that our method is different from theirs. For
example, the approach of ~\cite{Cilleruelo} makes use of complex analysis, but
ours is purely algebraic. The new ingredients in our proof are
fractional power series and their logarithmic derivatives.

\section{The idea of proof: Translation of the problem into generating functions}
For every set $\mathcal A$ of non-negative integers, the generating function of $\mathcal A$ is the formal power series $f_{\mathcal A}(x)$
defined as $$f_{\mathcal A}(x)=\sum \limits _{a\in \mathcal A}x^a.$$
In this way, the subsets of non-negative integers are in one to one correspondence to formal power series with coefficients 0 or 1.

The power series
$f_{\mathcal A}(x)$ also defines an analytic function around $x=0$. Indeed, if $\mathcal A$ is finite, then $f_{\mathcal A}(x)$ is a polynomial.
Otherwise,  $f_{\mathcal A}(x)$ has radius of convergence $r=1$ at $x=0$.

We now translate the combinatorial problem in the language of
generating functions. Let $\mathcal A$ be a subset of non-negative
integers and $${\mathrm M}=\{(b_0,e_0),(b_1,e_1),\cdots
,(b_m,e_m)\}.$$ The following equation is fundamental: \vskip 0.1cm
\begin{equation}\label{I}
\begin{aligned}
&f_{\mathcal A}(x^{b_0})^{e_0}f_{\mathcal A}(x^{b_1})^{e_1}\cdots f_{\mathcal A}(x^{b_m})^{e_m}\\
=&\sum \limits  _{a_{i,j}\in \mathcal A}x^{b_0(a_{0,1}+\cdots +a_{0,e_0})+\cdots +b_m(a_{m,1}+\cdots +a_{m,e_m})}\\
=&\sum \limits _{n=0}^{+\infty }r_{{\mathrm M}}(n,\mathcal A)x^n.\\
\end{aligned}
\end{equation}

Assume $r_{\mathrm M}(n,\mathcal A)=c \ne 0$ for $n>n_0$. Then
\begin{equation}\label{II}
\begin{aligned}
\sum \limits _{n=0}^{+\infty}r_{\mathrm M}(n,\mathcal A)x^n&=\sum \limits _{n=0}^{n_0}r_{\mathrm M}(n,\mathcal A)x^n+\sum \limits _{n=n_0+1}^{+\infty}cx^n\\
&=\sum \limits _{n=0}^{n_0}r_{\mathrm M}(n,\mathcal A)x^n+\frac {\displaystyle cx^{n_0+1}}{\displaystyle 1-x}=\frac
{\displaystyle P(x)}{\displaystyle 1-x},\\ \end{aligned}
\end{equation}
where $P(x)$ is a polynomial in $\mathbb Z[x]$ with
$P(1)\ne 0.$ Notice that $P(1)\ne 0$ is equivalent to $c\ne 0$.

Combining (\ref {I}) and (\ref {II}), we see that if $r_{\mathrm M}(n,\mathcal A)=c \ne 0$ for $n$ large enough, then
$f(x)=f_{\mathcal A}(x)$ is a solution in $\mathbb Z[[x]]$ of the
following equation
\begin{equation}\label{III}
f(x^{b_0})^{e_0}f(x^{b_1})^{e_1}\cdots f(x^{b_m})^{e_m}=\frac
{\displaystyle P(x)}{\displaystyle 1-x}
\end {equation}
for some $P(x)\in \mathbb Z[x]$ with $P(1)\ne 0.$

Conversely, a solution $f(x)\in \mathbb Z[[x]]$ of (\ref {III}) with
coefficients in $\{0,1\}$ defines, by the relation
$f(x)=f_{\mathcal A}(x)$, a subset $\mathcal A$ such that
$$r_{\mathrm M}(n,\mathcal A)
=c\ne0$$
for $n$ large enough.

Summing up, we get the following lemma.
\begin{lemma}\label{1} There exists an infinite subset $\mathcal A$ of non-negative integers
such that $r_{\mathrm M}(n, A)$ is a nonzero constant for $n$ large
enough if and only if there is a polynomial $P(x)\in \mathbb Z[x]$
with $P(1)\ne 0$ such that (\ref {III}) has a solution $f(x)\in
\mathbb Z[[x]]$ with coefficients $\in \{0,1\}$.
\end{lemma}

It is convenient to work with power series with constant
term being 1. For subset $\mathcal A$ of non-negative integers, this can be
achieved by replacing $\mathcal A$ by $\mathcal A- {\rm min}\{x|x\in
\mathcal A\}$. Alternatively, for power series $f(x)$, this can be
achieved by dividing $f_{\mathcal A}(x)$ by the lowest term. Obviously, this does not affect the solvability of equation
(\ref{III}). So, from now on, we always assume $0\in \mathcal A$ and $f(0)=1$ unless specification.

By lemma \ref{1} and the above arguments, to prove Theorem \ref{I}, it is sufficient to prove the following theorem.
\begin{theorem}\label{mainnew}
Let $m\ge 1$. Assume there exists a prime $p$ and a positive integer $t$ such that
$p^t\mid b_0,\ {\rm but}\ p^t\nmid b_{i}\ {\rm for}\ 1\leq i\leq m.
$ Then, for any $P(x)\in \mathbb Z[x]$ with $P(0)=1$ and
$P(1)\ne 0,$ the equation
 \begin{equation}\label{eqnew}
f(x^{b_0})^{e_0}f(x^{b_1})^{e_1}\cdots f(x^{b_m})^{e_m}=\frac
{\displaystyle P(x)}{\displaystyle 1-x}
\end {equation}
has no solution in $\mathbb
C[[x]].$\end{theorem}

In the following, we will illustrate the idea of the proof, especially the motivation of fractional power series. Let us look at Moser's argument first.

Moser's argument: For each $k\ge 2$, Moser constructed an infinite set $\mathcal A$ such that $r_{1,k}(n, \mathcal A)=1$ for all $n\ge 0$ by solving the equation
$$f(x)f(x^k)=\sum _{n\ge 0}^{+\infty}x^n=\frac {1}{1-x}.$$
Writing it as
$$f(x)=\frac {1}{1-x}f(x^k)^{-1},$$
then iterate
\begin {equation*}\begin {aligned}
f(x)&=\frac {\displaystyle 1}{\displaystyle 1-x}\left(\frac {\displaystyle 1}{\displaystyle 1-x^k}\right)^{-1}f(x^{k^2})\\
&=\frac {\displaystyle 1}{\displaystyle 1-x}\left(\frac {\displaystyle 1}{\displaystyle 1-x^k}\right)^{-1}\left(\frac {\displaystyle 1}{\displaystyle 1-x^{k^2}}\right)f(x^{k^3})^{-1}\\
&=\cdots \cdots \\
&=\prod \limits _{i=0}^{j-1}\left(\frac {\displaystyle 1}{\displaystyle 1-x^{k^i}}\right)^{(-1)^i}\cdot f(x^{k^j})^{(-1)^j}.\\
\end {aligned} \end {equation*}
Letting $j\rightarrow +\infty$, we get
$$f(x)=\prod _{i=0}^{+\infty }\left(\frac {1}{1-x^{k^i}}\right)^{(-1)^i}=\prod _{i=0}^{+\infty}(1+x^{k^{2i}}+x^{2k^{2i}}+\cdots +x^{(k-1)k^{2i}}).$$
By the uniqueness of $k$-adic representation of an integer, $f(x)$
is the generating function of the set $$\mathcal A=\left \{\sum
\limits _{i=0}^{+\infty}\varepsilon _{i}k^{2i},\varepsilon _i\in
\{0,1,\cdots ,k-1\}\right\}.$$

Our initial approach is similar to Moser's argument.  For simplicity, we will take the
example, ${\mathrm M}=\{(2,1),(3,1)\}$ to illustrate the ideas. From Lemma \ref{1}, we need to consider the equation
\begin{equation}\label{IV}
f(x^2)f(x^3)=\frac {P(x)}{1-x}.\end{equation} with $P(0)=1$ and
$P(1)\ne 0$.

If Equation (\ref{IV}) has a solution $f=f_{\mathcal A}$ for some
infinite subset $\mathcal A$ of non-negative integers, then $f$
defines an analytic function in the unit disk with $f(0)=1$.

Let $0<x<1$ and $x^{\frac {1}{2}}$ be the positive square root of
$x$. Substituting $x$ by $x^{\frac {1}{2}}$ in (\ref {IV}), we get
\begin {equation*}
f(x)=\frac {P(x^{\frac {1}{2}})}{1-x^{\frac {1}{2}}}f(x^{\frac {3}{2}})^{-1}.
\end{equation*}
Repeating Moser's arguments, we get
\begin{equation*}
f(x)=\prod _{k=0}^{j-1}\left (\frac {P(x^{\frac {1}{2}(\frac {3}{2})^k})}{1-x^{\frac {1}{2}(\frac {3}{2})^k}}\right )^{(-1)^k}f(x^{(\frac {3}{2})^j})^{(-1)^j}.
\end{equation*}
Since $x^{\frac {1}{2}(\frac {3}{2})^j}\rightarrow 0$ and $f(x^{\frac {1}{2}(\frac {3}{2})^j})\rightarrow 1$ as $j\rightarrow +\infty$, we obtain
\begin {equation}\label{V}
f(x)=\prod _{k=0}^{+\infty }\left (\frac {P(x^{\frac {1}{2}(\frac {3}{2})^k})}{1-x^{\frac {1}{2}(\frac {3}{2})^k}}\right)^{(-1)^k},
\end{equation}
for $0<x<1.$

Viewing
$x^{\frac {1}{2}(\frac {3}{2})^k}$ as an analytic function defined
in $\mathbb C-(-\infty ,0]$ with value $1$ at $x=1$, since for any
positive integer $n$,
\begin{equation*}
\sum _{k=0}^{+\infty }|x^{n\frac {1}{2}(\frac {3}{2})^k}|\le \sum
_{k=0}^{+\infty }r^{n\frac {1}{2}(\frac {3}{2})^k}<+\infty,\ {\rm
if}\ |x|\le r<1, \end{equation*} the infinite products
\begin{equation*}
\prod _{k=0}^{+\infty}P(x^{\frac {1}{2}(\frac {3}{2})^{2k}}),\ \prod _{k=0}^{+\infty}P(x^{\frac {1}{2}(\frac {3}{2})^{2k+1}}),\
\prod _{k=0}^{+\infty}(1-x^{\frac {1}{2}(\frac {3}{2})^{2k}}),\ {\rm and}\ \prod _{k=0}^{+\infty}(1-x^{\frac {1}{2}(\frac {3}{2})^{2k+1}})
\end{equation*}
are absolutely and uniformly convergent in compact subsets of $$D'=\{x\mid x\in\mathbb
C,\ |x|< 1\}-\{x\mid x\in\mathbb R,\ -1<x\le 0\},$$ hence,
analytic in $D'$ (e.g., see Proposition 3.2 of Chapter 5
of~\cite{Stein}). Therefore, the right hand side of (\ref {V}) is a
meromorphic function in $D'$.

The analytic function $f$ is determined by its values on the
interval $(0,1)$ (e.g., see Corollary 4.9 of Chapter 2
of~\cite{Stein}). Therefore, by (\ref {V}),
\begin {equation}\label{Vnew}
f(x)=\prod _{k=0}^{+\infty }\left (\frac {P(x^{\frac {1}{2}(\frac {3}{2})^k})}{1-x^{\frac {1}{2}(\frac {3}{2})^k}}\right)^{(-1)^k}
\end{equation}
holds for all $x\in D'.$

However, from Equation ({\ref{Vnew}}),
it seems that $f(x)$ can not be analytic around 0. This contradicts
to the hypothesis that $f(x)=f_{\mathcal A}(x)$, which is analytic in the unit disk. The rigorous proof goes as follows.

A useful method to treat ``infinite products" is taking its
logarithmic derivative, which transforms ``infinite products" to
``infinite sums" (e.g., see Proposition 3.2 of Chapter 5
of~\cite{Stein}). So instead of considering $f(x)$, we look at
$\frac {\displaystyle f'(x)}{\displaystyle f(x)}$. Since $f(0)=1$,
if $f(x)$ is analytic around zero, so is $\frac {\displaystyle
f'(x)}{\displaystyle f(x)}$.

As $P(0)=1$, we can assume $$\frac {\displaystyle P(x)}{\displaystyle 1-x}=\prod \limits _{i}(1-\alpha _ix)^{n_i}, {\rm with}\ \alpha_i '{\rm s\  distict.}$$

Denote $$\frac {\displaystyle P(x)}{\displaystyle 1-x}=G(x).$$ Then
\begin {equation}\label{VI}
x\frac {G'(x)}{G(x)}=\sum _{i}n_i\frac {-\alpha _ix}{1-\alpha _i x}=-\sum _{n=1}^{+\infty }\sum _{i}n_i\alpha _i^nx^n.
\end{equation}
From equation (\ref{Vnew}) and (\ref{VI}), we have\begin {equation}\label{VII}
\begin {aligned}
x\frac {f'(x)}{f(x)}&=\sum _{k=0}^{+ \infty }(-1)^k\frac {G(x^{\frac {1}{2}(\frac {3}{2})^k})'x}{G(x^{\frac {1}{2}(\frac {3}{2})^k})}
\\&=\sum _{k=0}^{+ \infty }(-1)^k\frac {1}{2}(\frac {3}{2})^k \frac {G'(x^{\frac {1}{2}(\frac {3}{2})^k})}{G(x^{\frac {1}{2}(\frac {3}{2})^k})}x^{\frac {1}{2}(\frac {3}{2})^k}\\
&=-\sum _{k=0}^{+\infty }(-1)^k \frac {1}{2}(\frac {3}{2})^k \sum _{n=1}^{+\infty }\sum _{i}n_i\alpha _i^nx^{n\frac {1}{2}(\frac {3}{2})^k}\\
\end {aligned}
\end {equation}

Note that since $|x|<1,\ x^{\frac {1}{2}(\frac {3}{2})^k}$ goes to zero
very fast as $k\rightarrow +\infty.$ A routine argument, which we do
not make here, shows that (\ref {VII}) are absolutely and uniformly
convergent, in a small neighborhood of zero (inside $D'$).
Therefore, we can take derivatives of (\ref {VII}) term by term
(e.g., see Theorem 5.3 of Chapter 2 of~\cite{Stein}).

Taking derivatives of (\ref {VII}) of all order term by term and evaluating the
derivatives at zero, one can see $\frac {\displaystyle
xf'(x)}{\displaystyle f(x)}$ is analytic around zero if and only if
the coefficient of $x^{\lambda }$ in (\ref {VII}) is zero, whenever
$\lambda \not \in \mathbb N$.

Letting the coefficient of $x^{\lambda }$ in (\ref {VII}) be 0, we get the equation
\begin {equation}\label{VIII}
\sum _{\frac {1}{2}(\frac {3}{2})^k|\lambda }(-1)^k\frac
{1}{2}(\frac {3}{2})^k\sum _{i}n_i\alpha _i^{\lambda \cdot 2(\frac
{2}{3})^k}=0,
\end {equation}
where
$$\frac {1}{2}(\frac {3}{2})^k|\lambda\Leftrightarrow\lambda \cdot 2(\frac {2}{3})^k\in \mathbb N.
$$
Finally, we succeed to prove for all $\lambda \not \in \mathbb N$, Equations (\ref {VIII}) have no common solution $\alpha _i 's$.

 After that, we realized
that $f(x)$ and $x\frac {\displaystyle f'(x)}{\displaystyle f(x)}$ (see Equations (\ref{Vnew}), (\ref{VII}))
can be viewed as some generalized formal series, which we call
fractional power series. Then everything can be computed formally in the
ring of fractional power series. In the rest of paper, we will use fractional power series other than analytic functions since the convergence of the former ones are much simpler than the latter ones.

The above arguments explain the motivation of using fractional power
series. As far as we know, the notion of fractional power series do
not appear in the literature. So they will be defined and discussed
in detail in Section 3. Generally speaking, fractional power series
behave like formal power series.

After the preparation of section 3, we begin to prove the main
result of this paper, Theorem \ref{mainnew}. The proof is actually
direct, but it is rather long. So We had better divide it into
several steps. The plan of the proof will be described in detail at the beginning
of section 4, after we introduce some basic notations. Section 4 and
section 5 provide all the ingredients of the proof. Finally, we prove
Theorem \ref{mainnew} in section 6.

At last, we discuss the question of S\'{a}rkozy and S\'{o}s in
section 7. We will give a conjectural answer in the case that all
the coefficients of linear forms are positive.

\section {Fractional power series}
In this section, we introduce the concept of fractional power series and basic operations of them, including their convergence, derivatives,
infinite products and logarithmic derivatives, etc.
We also prove their basic properties. These are fundamental to our later computations.

Let $\theta _1, \cdots ,\theta _m>1$ be distinct real numbers and
$b\in \mathbb N$ be a positive integer. Define $\mathbb Z_{\ge
0}[x_1,\cdots ,x_m]$ to be the set of polynomials of $x_1,\cdots x_m$
with coefficients of non-negative integers. Define
$$\Lambda =\left \{\frac {1}{b}F(\theta _1,\cdots ,\theta _m)\mid F\in \mathbb Z_{\ge 0}[x_1,\cdots ,x_m]\right \}.$$
We call $\Lambda $ the lattice associated to $(b; \theta _1,\cdots
,\theta _m)$.
\begin{proposition}\label{p1} Let $\Lambda $ be defined as above. Then
\begin{equation*}\begin{aligned}
&(1)\ \Lambda{\rm\ is\ discrete,\ i.e.,}\ \forall\ M>0,\ \{\lambda\in \Lambda\mid \lambda<M\}\ {\rm is\ a\ finite\ set.}\\
&(2)\ {\rm If\ } \lambda ,\lambda '\in \Lambda,\ {\rm then\ }  \lambda +\lambda
'\in \Lambda .\\
&(3)\ {\rm If\ } \lambda \in \Lambda,\ {\rm then\ }   \theta _i\lambda \in \Lambda\
 {\rm for}\ i=1,\cdots ,m.\\
&(4)\ \mathbb Z _{\ge 0}\subseteq \Lambda,\ {\rm  where\ } \mathbb Z _{\ge
0}\ {\rm is\ the\ set\ of\ nonnegative\ integers.} \ \ \ \ \ \ \ \ \
\end{aligned}\end{equation*}
\end{proposition}
\begin{proof}
We only prove (1). Let $F\in \mathbb Z_{\ge 0}[x_1,\cdots ,x_m]$.
Denote the total degree of $F$ with respect to $x_1,\cdots ,x_m$ by
$d$. Let $C$ be an arbitrary coefficient of $F$. If
\begin{equation}\label{tirednew}
\frac
{\displaystyle 1}{\displaystyle b}F(\theta _1 ,\cdots ,\theta _m)\le
M,\end{equation} then
$$(\min \{\theta _1,\cdots ,\theta _m\})^d\le F(\theta _1,\cdots ,\theta
_m)\le bM,\ $$$$0\le C\le F(\theta _1,\cdots ,\theta _m)\le bM.$$ So
there are only finitely many $F$ satisfying Equation (\ref{tirednew}).
\end{proof}
\begin{definition}\label{d1}
The formal series $$\sum \limits _{\lambda \in \Lambda
}c_{\lambda}x^{\lambda },\ {\rm with}\ c_{\lambda }\in \mathbb C, $$ are called
fractional power series with respect to $\Lambda $, or $(b; \theta
_1,\cdots , \theta _m).$ Define $$\sum \limits _{\lambda \in
\Lambda }c_{\lambda}x^{\lambda }=\sum \limits _{\lambda \in
\Lambda }c_{\lambda}'x^{\lambda }$$ if and only
if $c_{\lambda}=c_{\lambda}'$ for all $\lambda \in \Lambda .$ Denote $\mathbb C[[x^{\Lambda}]]$ to be the set of all fractional power
series with respect to $\Lambda $.
\end{definition}
The following definition makes $\mathbb C[[x^{\Lambda}]]$ a commutative ring with unit element.
\begin{definition}\label{d2}
For $$\sum \limits _{\lambda \in \Lambda }c_{\lambda}x^{\lambda },\
\sum \limits _{\lambda \in \Lambda }c_{\lambda}'x^{\lambda }\in
\mathbb C[[x^{\Lambda}]],$$ their sum and product are defined as
\begin{equation*}
\begin {aligned}
\sum  _{\lambda \in \Lambda }c_{\lambda}x^{\lambda }+\sum _{\lambda
\in \Lambda }c_{\lambda}'x^{\lambda }&=\sum
_{\lambda \in \Lambda }(c_{\lambda}+c_{\lambda}')x^{\lambda },\\
\sum _{\lambda \in \Lambda }c_{\lambda}x^{\lambda }\cdot \sum
_{\lambda \in \Lambda }c_{\lambda}'x^{\lambda }&=\sum _{\lambda \in
\Lambda }(\sum _{\mu +\nu =\lambda }c_{\mu}c_{\nu
}')x^{\lambda }. \\
\end {aligned}
\end{equation*}We call $\mathbb C[[x^{\Lambda }]]$
the ring of fractional power series with respect to $$\Lambda\ {\rm or}\ (b;\theta _1,\cdots ,\theta _m).$$
\end{definition}
By (3) of Proposition \ref{p1}, we know that $${\rm for\ any}\ \lambda \in \Lambda,\ {\rm the\ sum\ } \sum \limits _{\mu +\nu =\lambda }c_{\mu}c_{\nu }$$ is a finite
sum, so the multiplication of two elements of $\mathbb C[[x^{\Lambda }]]$ is well-defined.
It is easily seen that $x^0$ is the unit element of $\mathbb C[[x^{\Lambda }]]$.  Moreover, by (4) of Proposition \ref {p1}, we have
$$\mathbb C[[x]]\subseteq \mathbb C[[x^{\Lambda }]].$$
\begin{remark}\label{r1}
If $\{\theta _1,\cdots ,\theta _m\}=\varnothing $, then
$$\mathbb C[[x^{\Lambda }]]=\{\sum \limits _{n\ge 0}c_nx^{\frac {n}{b}}\mid c_n\in \mathbb C\}=\mathbb C[[x^{\frac {1}{b}}]].$$
 If $f\in \mathbb C[[x^{\frac {1}{b}}]]$ for some $b\geq 1$, then $f$ is called a fractional power serie by Stanley
 (see page 161 of ~\cite{Stanley}). So Definition \ref{d1} can be viewed as a generalization of Stanley's definition.
\end{remark}
Generally speaking, $\mathbb C[[x^{\Lambda }]]$ has many properties
similar to  $\mathbb C[[x]]$. For example, we can define metrics
both on $\mathbb C[[x^{\Lambda }]]$ and  $\mathbb C[[x]]$, which
make them complete metric spaces.
\begin{definition}\label{d3}
Let $f=\sum \limits _{\lambda \in \Lambda}c_{\lambda }x^{\lambda}\in
\mathbb C[[x^{\Lambda }]]$. The order of $f$, denoted by ${\rm ord}
f$, is defined as follows:
\begin{equation*}
{\rm ord}f=\left \{
\begin{aligned}
\min _{c_{\lambda}\ne 0}\{\lambda \},&&{\rm if}\ f\ne 0;\\
+\infty,&&{\rm if}\ f=0.\\
\end{aligned}
\right.
\end{equation*}
\end{definition}
We have the following proposition.
\begin{proposition}\label{p2}
Let $f, g\in \mathbb C[[x^{\Lambda }]]$. Then
\begin{equation*}\begin{aligned}
&(1)\ {\rm ord}(f+g)\ge \min \{{\rm ord}f,{\rm ord}g\}.\\
&(2)\ {\rm ord}(f\cdot g)={\rm ord}f+{\rm ord}g.\\
&(3)\ {\rm ord}f=+\infty \Leftrightarrow f=0. \ \ \ \ \ \ \ \ \ \ \ \ \ \ \ \ \ \ \ \ \ \ \ \ \ \ \ \ \ \ \ \ \ \ \ \  \ \ \  \ \ \ \ \ \ \ \ \ \ \ \ \ \ \ \ \ \ \ \ \ \ \ \  \ \ \ \ \ \
\end{aligned}\end{equation*}
\end{proposition}
Fix some real number $\beta \in (0,1)$.
\begin{definition}\label{d4}
Let $f \in \mathbb C[[x^{\Lambda }]]$. The valuation of $f$, denoted by $|f|$, is equal to $\beta ^{{\rm ord }f}$.
\end{definition}
Corresponding to Proposition \ref {p2}, we have
\begin {proposition}\label{p3}
Let $f,g\in \mathbb C[[x^{\Lambda }]]$. Then
\begin{equation*}\begin{aligned}
&(1)\ |f+g|\le \max \{|f|,|g|\}.\\
&(2)\ |f\cdot g|=|f||g|.\\
&(3)\ |f|=0\Leftrightarrow f=0. \ \ \ \ \ \ \ \ \ \ \ \ \ \ \ \ \ \ \ \ \ \ \ \ \ \ \ \ \ \ \ \ \ \ \ \  \ \ \  \ \ \ \ \ \ \ \ \ \ \ \ \ \ \ \ \ \ \ \ \ \ \ \  \ \ \ \ \ \
\end{aligned}\end{equation*}
\end{proposition}
Given two elements $f,g\in \mathbb C[[x^{\Lambda }]]$, their distance is defined as
$$d(f,g)=|f-g|.$$

By Proposition \ref {p3}, $ \mathbb C[[x^{\Lambda }]]$ is really a metric space. And (1) of Proposition \ref {p3}
is usually called the strong triangle inequality.

A sequence $\{f_n\}_{n\in \mathbb N}$ is convergent to $f$ if and
only if $$\lim \limits _{n\rightarrow +\infty}d(f_n,f)=0, $$ or
equivalently, $$\lim \limits _{n\rightarrow +\infty} {\rm
ord}(f_n-f)=+\infty.$$ In this case, we denote $$f=\lim \limits
_{n\rightarrow +\infty}f_n.$$

The following proposition shows that $ \mathbb C[[x^{\Lambda }]]$ is a complete metric space, i.e., every cauchy sequence converges.
\begin{proposition}\label{p4}
If $\lim \limits _{n\rightarrow +\infty}{\rm ord}(f_{n+1}-f_n)=+\infty$, then there exists $f\in \mathbb C[[x^{\Lambda }]]$
such that $f=\lim \limits _{n\rightarrow +\infty}f_n$.
\end{proposition}
\begin{proof}
Let $$f_n=\sum \limits _{\lambda \in
\Lambda}c_{n,\lambda}x^{\lambda}.$$
Since $$\lim \limits _{n\rightarrow +\infty} {\rm
ord}(f_{n+1}-f_n)=+\infty,$$for any $\lambda \in \Lambda$,
there exists $N\in \mathbb N$ such that
$${\rm ord}(f_{n+1}-f_n) >\lambda\ {\rm for}\ n>N.$$
This implies that
$$c_{n+1,\lambda}=c_{n,\lambda},\ {\rm when}\ n>N,$$ that is, for $\lambda$ being fixed, the sequence
$c_{n,\lambda}$ is constant for $n$ large enough.
Therefore, let $$f=\sum
\limits _{\lambda \in \Lambda} (\lim \limits _{n\rightarrow
+\infty}c_{n,\lambda})x^{\lambda}.$$ Then ${\rm ord}(f-f_n)>\lambda$
if $n>N$. Therefore $$\lim \limits _{n\rightarrow +\infty}f_n=f.$$
\end{proof}
For $f=\sum \limits _{\lambda \in \Lambda}c_{\lambda }x^{\lambda}\in \mathbb C[[x^{\Lambda }]]$, denote $c_0$  by $f(0)$.
\begin{corollary}\label{c1}
$f$ is invertible if and only if $f(0)\ne 0$.
\end{corollary}
\begin{proof}
If there exists $g\in \mathbb
C[[x^{\Lambda }]]$ such that $f\cdot g=1$, then $f(0)g(0)=1$. Therefore,
$f(0)\ne 0$.

Conversely, assume $f(0)\ne 0$. Write $$f=f(0)(1+h)\
{\rm with}\ {\rm ord} (h)>0.$$ By Proposition \ref {p4}, $$1+\sum \limits
_{i=1}^{+\infty}(-1)^ih^i$$ converges. Therefore, $$f^{-1}=f(0)^{-1}(1+\sum
\limits _{i=1}^{+\infty }(-1)^ih^i)\in \mathbb C[[x^\lambda]].$$

\end{proof}
\begin{corollary}\label{newc1}
For $n\ge 1$, assume ${\rm ord} f_n>0$ and $\lim \limits
_{n\rightarrow +\infty}{\rm ord}f_n=+\infty$. Then the infinite
product $\prod \limits _{n=1}^{+\infty}(1+f_n)$ converges.
\end{corollary}
\begin{proof}
Since $$\prod \limits _{n=1}^{m+1}(1+f_n)-\prod \limits
_{n=1}^{m}(1+f_n)=f_{m+1}\prod \limits _{n=1}^{m}(1+f_n),$$ its order
equals to ${\rm ord} f_{m+1}$. By the assumption, $$\lim \limits
_{m\rightarrow +\infty}{\rm ord}(\prod \limits
_{n=1}^{m+1}(1+f_n)-\prod \limits _{n=1}^{m}(1+f_n))=+\infty.$$ By
Proposition \ref {p4}, we get the desired result.

\end{proof}
Let $$f(x)=\sum \limits _{\lambda \in \Lambda}c_{\lambda
}x^{\lambda}\in \mathbb C[[x^{\Lambda }]].$$ Then $$\sum \limits
_{\lambda \in \Lambda}c_{\lambda }x^{\lambda \theta _i}\in \mathbb
C[[x^{\Lambda }]],$$ by (3) of Proposition \ref {p1}, where
$i=1,\cdots ,m$. Then define $$f(x^{\theta _i})=\sum \limits _{\lambda \in
\Lambda}c_{\lambda }x^{\lambda \theta _i}.$$ This can be viewed as changing variable ``$x$" by ``$x^\theta$". It
is easy to see that the map $f(x)\mapsto f(x^{\theta _i})$ is a continuous ring
homomorphism of $\mathbb C[[x^{\Lambda }]]$.
\begin{definition}\label{d5}
Let $$f(x)=\sum \limits _{\lambda \in \Lambda}c_{\lambda }x^{\lambda}\in \mathbb C[[x^{\Lambda }]].$$ Define the derivative of $f$ by $$xf'(x)=
\sum \limits _{\lambda \in \Lambda}\lambda c_{\lambda }x^{\lambda}\in \mathbb C[[x^{\Lambda }]].$$
\end{definition}
\begin{remark}\label{r2}
In Definition \ref {d5}, we multiply the usual derivative $f'$ by
$x$ to make sure $xf'\in \mathbb C[[x^{\Lambda}]]$.
\end{remark}
\begin{proposition}\label{p5}
Let $f,\ g\in \mathbb C[[x^{\Lambda }]]$. Then
\begin{equation*}\begin{aligned}
&(1)\ x(f+g)'=xf'+xg'.\\
&(2)\ x(f\cdot g)'=xf'\cdot g +f\cdot xg'.\\
&(3)\ x(f(x^{\theta _i}))'=\theta _i(xf')(x^{\theta _i}),\ {\rm for}\
i=1,\cdots ,m.\\
&(4)\ \lim\limits _{n\rightarrow +\infty} (xf_n')=x(\lim \limits _{n\rightarrow +\infty}f_n)',\
{\rm if}\ \lim\limits _{n\rightarrow +\infty }f_n\ {\rm exists}. \ \ \ \ \ \ \ \ \ \ \ \ \ \ \ \ \ \ \ \ \ \ \ \ \ \ \ \ \ \ \ \ \ \ \ \  \ \ \  \ \ \ \ \ \ \
\end{aligned}\end{equation*}
\end{proposition}

\begin{definition}\label{d6}
Let $$f(x)=\sum \limits _{\lambda \in \Lambda}c_{\lambda }x^{\lambda}\in \mathbb C[[x^{\Lambda }]]\  {\rm with}\ f(0)\ne 0.$$ We call
$\frac {\displaystyle xf'}{\displaystyle f}$ the logarithmic derivative of $f$.
\end{definition}
\begin{remark}\label{r3}
If $f(0)=0$, from Corollary \ref{c1}, $f$ is not invertible. Hence,
$\frac {\displaystyle xf'}{\displaystyle f}$ may not belong to
$\mathbb C[[x^{\Lambda }]]$. This is different from the case of
$\mathbb C[[x]]$.
\end{remark}
By Proposition \ref {p5}, we have
\begin{proposition}\label{p6}
Let $f,\ g\in \mathbb C[[x^{\Lambda }]]$. Then
\begin{equation*}
\begin{aligned}
 &{\rm (1)}\ \frac {\displaystyle x(fg)'}{\displaystyle fg}
=\frac {\displaystyle xf'}{\displaystyle f} +\frac {\displaystyle xg'}{\displaystyle g}
\\
  &{\rm (2)}\ \frac {\displaystyle x(f(x^{\theta _i}))'}{\displaystyle
f(x^{\theta _i})} =\theta _i\left (\frac {\displaystyle
xf'}{\displaystyle f}\right ) (x^{\theta _i})\ {\rm for}\ i=1,\cdots , m.\ \ \ \ \ \ \ \ \ \ \ \ \ \ \ \ \ \ \ \ \
\end{aligned}
\end{equation*}
\end{proposition}
The following proposition shows that the logarithmic derivative transforms infinite products to infinite sums.
\begin{proposition}\label{p7}
For $n\ge 1$, assume ${\rm ord} f_n>0$ and $\lim \limits
_{n\rightarrow +\infty}{\rm ord }f_n=+\infty$. Then
\begin{equation*}
x(\prod \limits _{n=1}^{+\infty}(1+f_n))'(\prod \limits _{n=1}^{+\infty}(1+f_n))^{-1}
=\sum \limits _{n=1}^{+\infty} xf_n'(1+f_n)^{-1}.
\end{equation*}
\end{proposition}
\begin{proof}
By (1) of Proposition \ref {p6}, we have
\begin{equation}\label{IX}
x(\prod _{n=1}^{N}(1+f_n))'\cdot (\prod _{n=1}^{N}(1+f_n))^{-1}=\sum _{n=1}^{N}xf_n'(1+f_n)^{-1}.
\end{equation}
By (4) of Proposition \ref {p5}, we get
\begin{equation*}
\lim \limits _{N\rightarrow +\infty}x(\prod \limits _{n=1}^{N}(1+f_n))'=x(\prod \limits _{n=1}^{+\infty}(1+f_n))'.
\end{equation*}
Since
$$
\lim \limits _{N\rightarrow +\infty}\prod \limits _{n=1}^{N}(1+f_n)^{-1}=(\prod \limits _{n=1}^{+\infty}(1+f_n))^{-1},
$$
letting $N\rightarrow +\infty$ in (\ref {IX}), we get the desired result.
\end{proof}
\begin{proposition}\label{p8}
Let $f\in \mathbb C[[x^{\Lambda}]]$ with $f(0)\ne 0$. Then $$f\in \mathbb C[[x]]\Leftrightarrow
\frac {\displaystyle xf'}{\displaystyle f}\in \mathbb C[[x]].$$
\end{proposition}
\begin{proof}
We only prove the ``if" part. Assume $$\frac {\displaystyle
xf'}{f}=\sum \limits _{n=1}^{+\infty}c_nx^n\in \mathbb C [[x]].$$ Let $$g=\exp (\sum
\limits _{n=1}^{\infty}\frac {\displaystyle c_n}{n}x^n) =\sum
\limits _{m=0}^{+\infty}\frac {\displaystyle 1}{\displaystyle m!}
(\sum \limits _{n=1}^{\infty}\frac {\displaystyle c_n}{n}x^n)^m.$$
By (4) of Proposition \ref {p5}, we have
\begin{equation*}
xg'=\exp(\sum \limits _{n=1}^{+\infty}\frac {\displaystyle c_n}{\displaystyle n}x^n)
\sum \limits _{n=1}^{+\infty}c_nx^n.
\end{equation*}
Therefore, $$\frac {\displaystyle xg'}{\displaystyle g}=\sum \limits _{n=1}^{+\infty}c_nx^n=\frac {\displaystyle
xf'}{\displaystyle f}.$$ Since $$x(\frac {\displaystyle
g}{\displaystyle f})'(\frac { g}{ f})^{-1}
=\frac {\displaystyle xf'}{\displaystyle f}-\frac {\displaystyle
xg'}{\displaystyle g}=0,$$ we get $\frac{\displaystyle g}{\displaystyle f}$ is a constant. Thus $f\in \mathbb C [[x]]$.
\end{proof}
 The following power series are well-known.
 \begin{equation*}
 \begin{aligned}
 \exp (x)&=\sum \limits _{n=0}^{+\infty}\frac {x^n}{n!},\\
 \log (1+x)&=\sum \limits _{n=1}^{+\infty}\frac {(-1)^{n-1}x^n}{n},\\
 (1+x)^{\alpha }&=\sum \limits _{n=0}^{+\infty}\left ( \begin {aligned}\alpha \\ n \\ \end {aligned}
 \right )x^n,\ {\rm where} \ \alpha \in \mathbb C, \\{\rm and}\
 &\left ( \begin {aligned}\alpha \\ n \\ \end {aligned}
 \right )=\frac{\alpha (\alpha-1)\cdots(\alpha-n+1)}{n(n-1)\cdots
 1}.
 \end{aligned}
 \end{equation*}

 Let $f\in \mathbb C[[x^{\Lambda}]]$ with ${\rm ord}(f)>0$. Then
we can define $\exp(f),\ \log (1+f),\ (1+f)^{\alpha}$ by replacing
$x$ with $f$ in the above expressions. It is easy to see the
equalities which hold for $\exp(x),\ \log (x),\ (1+x)^{\alpha }$
also hold for $\exp(f),\ \log (1+f),\ (1+f)^{\alpha}$. For example,
we have
\begin{equation*}
 \begin{aligned}
((1+f)^{\alpha })^{\beta }&=(1+f)^{\alpha \beta},\\
(1+f)^{\alpha }(1+f)^{\beta }&=(1+f)^{\alpha +\beta},\\
(1+f)^{\alpha }(1+g)^{\alpha }&=(1+f+g+fg)^{\alpha },\\
x\left((1+f)^{\alpha }\right)'&=\alpha(1+f)^{\alpha-1 }xf'
\end{aligned}
\end{equation*}
where ${\rm ord}(f),\ {\rm ord} (g)>0,\ {\rm and}\ \alpha ,\beta \in \mathbb
C$.

Finally, we prove the following proposition.
\begin{proposition}\label{p9}
For $n\ge 1$, let $\alpha _n,\ \beta _n \in \mathbb C$. Then
\begin{equation}\label{X}
\prod _{n=1}^{+\infty }(1-x^n)^{\alpha _n}=\prod _{n=1}^{+\infty }(1-x^n)^{\beta _n}
\end{equation}
if and only if $\alpha _n=\beta _n$ for all $n\ge 1$.
\end{proposition}
\begin{proof}
Taking the logarithmic derivatives of (\ref{X}), we get
\begin{equation}\label{XI}
-\sum _{n=1}^{+\infty}\frac {\alpha _n nx^n}{1-x^n}=-\sum _{n=1}^{+\infty}\frac {\beta _n nx^n}{1-x^n}.
\end{equation}
Comparing the coefficients of the lowest terms of (\ref{XI}), we get
$\alpha _1=\beta _1$. Then
\begin{equation}\label{XI1}
-\sum _{n=2}^{+\infty}\frac {\alpha _n nx^n}{1-x^n}=-\sum
_{n=2}^{+\infty}\frac {\beta _n nx^n}{1-x^n}.
\end{equation}
Repeating the same procedure, we obtain $\alpha _2=\beta _2,\cdots
,\alpha _n=\beta _n, \cdots $. This concludes the proof.
\end{proof}

\section {Solving equation with fractional power series}
From now on, the following notations will be used unless specification.
\begin{equation*}
\begin{aligned}
&\bullet \ \ 2\le b=b_0<b_1<\cdots ,<b_m,\ {\rm positive \ integers.}\\
&\bullet \ \ \theta _1=\frac {b_1}{b_0},\cdots ,\theta _m=\frac {b_m}{b_0}.\\
&\bullet \ \ \Lambda ,\ {\rm the\ lattice\ associated\ to}\ (b;\ \theta _1,\cdots ,\theta _m).\\
&\bullet \ \ \mathbb C[[x^{\Lambda }]],\ {\rm the\ ring\ of\ fractional\ power\ series\ with\ respect\ to\ }\Lambda .\\
&\bullet \ \ e=e_0,e_1,\cdots ,e_m,\ {\rm positive \ integers}.\\
&\bullet \ \  \nu _1=\frac{e_1}{e_0},\cdots ,\nu_m=\frac {e_m}{e_0}.\\
&\bullet \ \ \mathbb N'=\{N\in \mathbb N\mid p|N\Rightarrow p|b_0b_1\cdots b_m\}\\
&\bullet \ \ \mathbb Q'=\{\frac {n}{m}\mid n,m\in \mathbb N'\}.\\
&\bullet \ \ \mathbb Q_b=\{\frac {n}{b^t}\mid n\in \mathbb N,\ t\in \mathbb Z,\ t\ge 0\}.\\
&\bullet \ \ \mathbb Q_b'=\{\frac {n}{b^t}\mid n\in \mathbb N',\ t\in \mathbb Z,\ t\ge 0\}.\\
&\bullet \ \ G(x)=\prod _{i=1}^{t}(1-\alpha _ix)^{n_i},\ {\rm with}\ \alpha _i '{\rm s\ distict,\ nonzero}.\\
&\bullet \ \ S=\{\alpha _i\mid \alpha _i^n=1\ {\rm for\ some}\ n\in \mathbb N',\ G(\frac {1}{\alpha _i})=0\}.\\
&\bullet \ \ H(x)=\prod _{\alpha _i\in S}(1-\alpha _ix)^{n_i}\ \ {\rm the}\ \mathbb N'{\rm -cyclotomic\ part\ of}\ G(x).\\
&\bullet \ \ P(x)\in \mathbb Z[x],\ P(0)=1,\ P(1)\ne 0.\\
&\bullet \ \ \lambda\mid\mu\Leftrightarrow\mu\lambda^{-1}\in \mathbb N,\ {\rm where}\ \lambda,\mu\in \mathbb Q,\ {\rm and}\ \lambda,\mu>0.\\
\end{aligned}
\end{equation*}

In this section, we will prove that the equation
\begin{equation}\label{XII}
f(x^{b_0})^{e_0}f(x^{b_1})^{e_1}\cdots f(x^{b_m})^{e_m}=G(x)
\end{equation}
has a unique solution $f(x)$ in $\mathbb C[[x^{\Lambda}]]$ with $f(0)=1$.

By taking the logarithmic derivative of $f$, we give a criterion for $f$ being a power series. As a corollary, we show if (\ref{XII}) has a
power series solution, then the  equation
\begin{equation}\label{XIII}
g(x^{b_0})^{e_0}g(x^{b_1})^{e_1}\cdots g(x^{b_m})^{e_m}=H(x)
\end{equation}
has a power series solution $g(x)$ with $g(0)=1$, where $H(x)$ is the $\mathbb N'$-cyclotomic part of $G(x)$.

Moreover, if $H(x)$
has the following form
\begin{equation}\label{XIV}
H(x)=\prod _{d\in \mathbb N'}(1-x^d)^{m_d},\ m_d\in \mathbb Z,\ m_d=0\ {\rm for}\ d\gg 0,
\end{equation}
then the power series solution of (\ref {XIII}) (if it exists) can be explicitly given by
\begin{equation}\label{XV}
g(x)=\prod _{d\in \mathbb N'}(1-x^d)^{g_d},\ g_d\in \mathbb Q.
\end{equation}

Finally, under some conditions on $b_0,b_1,\cdots,b_m$, we show
$g(x)$ is almost rational, that is, $$g_d=0\ {\rm for}\ d\gg 0,$$ in Equation
(\ref{XV}) (see Theorem \ref{t5}).

In the next section, under certain conditions on $H(x)$, we will
show that $g(x)$, the solution of (\ref{XIII}), can not be almost
rational (see Theorem \ref{main}). A contradiction!

This finally leads to the proof of Theorem \ref{mainnew} if we apply
the above results to
 the case
\begin{equation}\label{XV1}
G(x)=\frac {P(x)}{1-x}.
\end{equation}
and assume the existence of a prime $p$ and a positive integer $t$
such that $p^t\mid b_0,\ {\rm but}\ p^t\nmid b_{i}\ {\rm for}\ 1\leq
i\leq m. $

\begin{theorem}\label{t1}
The equation
\begin{equation}\label{XVI}
f(x^{b_0})^{e_0}f(x^{b_1})^{e_1}\cdots f(x^{b_m})^{e_m}=G(x)
\end{equation}
has a unique solution $f(x)\in \mathbb C[[x^{\Lambda }]]$ with $f(0)=1$. In fact,
\begin{equation}\label{XVII}
f(x)=\prod _{k=0}^{+\infty }\left (\prod _{1\le i_1,\cdots ,i_k\le m}G(x^{b^{-1}\theta _{i_1}\cdots ,\theta _{i_k}})^{e^{-1}\nu _{i_1}\cdots \nu _{i_k}}\right )^{(-1)^k}.
\end{equation}
\end{theorem}
\begin{proof}
{\bf Existence}: Recall $b=b_0,e=e_0$. Substituting $x$ by $x^{\frac
{1}{b}}$ and taking the $e$-th root of both sides of Equation (\ref
{XVI}), we get
\begin{equation}\label{XVIII}
f(x)f(x^{\theta _1})^{\nu _1}\cdots f(x^{\theta _m})^{\nu _m}=G(x^{\frac {1}{b}})^{\frac {1}{e}}.
\end{equation}
Then
\begin{equation}\label{XIX}
f(x)=G(x^{\frac {1}{b}})^{\frac {1}{e}}(f(x^{\theta _1})^{\nu _1}\cdots f(x^{\theta _m})^{\nu _m})^{-1}.
\end{equation}

Let $\mathcal F(f)=f(x^{\theta _1})^{\nu _1}\cdots f(x^{\theta
_m})^{\nu _m}$ and view it as an operator on $\mathbb C[[x^{\Lambda
}]]$. Obviously, $\mathcal F$ is multiplicative. Rewrite Equation (\ref{XIX}) as
 \begin{equation}\label{XX}
 f(x)=G(x^{\frac {1}{b}})^{\frac {1}{e}}\mathcal F(f)^{-1}.
 \end{equation}
 Then iterate£º
 \begin{equation}\label{XXI}
 \begin{aligned}
 f(x)&=G(x^{\frac {1}{b}})^{\frac {1}{e}}\mathcal F(G(x^{\frac {1}{b}})^{\frac {1}{e}})^{-1}\mathcal F^2(f)\\
 &=\cdots \cdots \\
 &=\prod _{k=0}^{n-1}\mathcal F^k(G(x^{\frac {1}{b}})^{\frac {1}{e}})^{(-1)^k}\mathcal F^n(f)^{(-1)^n}.\\
 \end{aligned}
 \end{equation}

 Since $$\mathcal F^n(f)=\prod \limits _{1\le i_1,\cdots ,i_n\le m}f(x^{\theta_{i_1} \cdots \theta_{i_n}})^{\nu _{i_1}\cdots \nu
 _{i_n}},$$
 we have $\lim \limits _{n\rightarrow +\infty}\mathcal F^n(f)=1.$

 Letting $n\rightarrow +\infty $ in Equation (\ref{XXI}),
 by Corollary \ref{newc1}, we get
 \begin{equation}\label{XXII}
 \begin{aligned}
 f(x)&=\prod _{k=0}^{+\infty }\mathcal F^k(G(x^{\frac {1}{b}})^{\frac {1}{e}})^{(-1)^k}\\
&=\prod _{k=0}^{+\infty }\left (\prod _{1\le i_1,\cdots ,i_k\le m}G(x^{b^{-1}\theta _{i_1}\cdots ,\theta _{i_k}})^{e^{-1}\nu _{i_1}\cdots \nu _{i_k}}\right )^{(-1)^k}.\\
 \end{aligned}
 \end{equation}
 Substituting (\ref{XXII}) into (\ref {XVIII}), we get
 \begin{equation*}\begin{aligned}f\cdot\mathcal F(f)=\prod _{k=0}^{+\infty}\mathcal F^k(G(x^{\frac {1}{b}})^{\frac
 {1}{e}})^{(-1)^k}\prod _{k=0}^{+\infty}\mathcal F^{k+1}(G(x^{\frac {1}{b}})^{\frac
 {1}{e}})^{(-1)^k}=G(x^{\frac {1}{b}})^{\frac
 {1}{e}}.
  \end{aligned}
 \end{equation*}
 So, (\ref{XXII}) is really a solution of (\ref{XVI}).

 {\bf Uniqueness}: Assume $f^{*}\in \mathbb C[[x^{\Lambda }]]$ is
 another solution of (\ref{XVI}) with $f^{*}(0)=1$, then
 \begin{equation}\label{XXIII}
 \frac {f}{f^*}(x^{b_0})^{e_0} \frac {f}{f^*}(x^{b_1})^{e_1}\cdots  \frac
 {f}{f^*}(x^{b_m})^{e_m}=1.
 \end{equation}
Write $$\frac {\displaystyle f}{\displaystyle
f^*}=1+c_{\mu}x^{\mu}+\sum \limits _{\lambda
>\mu }c_{\lambda }x^{\lambda}$$ with $c_{\mu }\ne 0$. Then the right
hand side of Equation (\ref{XXIII}) is $$1+e_0c_{\mu
}x^{b_0\mu}+``{\rm higher\ order\ terms}".$$ A contradiction. So
$f^*=f$.
\end{proof}
\begin{theorem}\label{t2}
The solution (\ref{XVII}) is a power series if and only if for any
 any $\lambda \in \mathbb Q_b'$ satisfying
$\lambda \not \in \mathbb N'$, any $u\in \mathbb N,\ (u,b_0\cdots
b_m)=1$, and any $\beta \in \mathbb C^{*}$, the following equation
\begin{equation}\label{XXIV}
\begin{aligned}
\sum _{k=0}^{+\infty} (-1)^k\sum _{b^{-1}\theta _{i_1}\cdots \theta
_{i_k}\mid \lambda} &\frac {\nu _{i_1}\cdots \nu_{i_k}}{e}\frac
{\theta _{i_1}\cdots \theta_{i_k}} {b}\\&\times\sum _{\alpha
_i^{\lambda b \theta _{i_1}^{-1}\cdots  \theta _{i_k}^{-1}b_0\cdots
b_m}=\beta } n_i\alpha _i^{\lambda ub \theta _{i_1}^{-1}\cdots
\theta _{i_k}^{-1}}=0 \end{aligned}\end{equation}holds, where
$b^{-1}\theta _{i_1}\cdots \theta _{i_k}\mid \lambda$ means their
quotient is a positive integer.
\end{theorem}
\begin{proof}
We will compute the logarithmic derivative of $f$ by equation
(\ref{XVII}).

Since $$G(x)=\prod \limits_{i}(1-\alpha _ix)^{n_i},$$ we have
\begin{equation}\label{XXV}
\frac {xG(x)'}{G(x)}=\sum _i\frac {-n_i\alpha _ix}{1-\alpha
_ix}=-\sum _{n=1}^{+\infty}(\sum _in_i\alpha _i^{n})x^n.
\end{equation}

By (2) of Proposition \ref{p6}, for $1\le i_1,\cdots ,i_k\le m$,
\begin{equation}\label{XXVI}
\frac {xG(x^{b^{-1}\theta _{i_1}\cdots \theta _{i_k}})'}{G(x^{b^{-1}\theta _{i_1}\cdots \theta _{i_k}})}=
-\frac {\theta _{i_1}\cdots \theta _{i_k}}{b}\sum _{n=1}^{+\infty}(\sum _in_i\alpha _i^n)x^{nb^{-1}\theta _{i_1}\cdots \theta _{i_k}}.
\end{equation}

By Proposition \ref {p7}, Equation (\ref{XVII}) and Equation (\ref{XXVI}),
\begin{equation}\label{XXVII}\begin{aligned}
\frac {xf'}{f} =\sum_{k=0}^{+\infty} (-1)^k&\sum _{1\le i_1,\cdots
,i_k\le m}\frac {\nu _{i_1}\cdots \nu_{i_k}}{e}\frac {xG(x^{b^{-1}\theta _{i_1}\cdots \theta _{i_k}})'}{G(x^{b^{-1}\theta _{i_1}\cdots \theta _{i_k}})}
\\=- \sum_{k=0}^{+\infty} (-1)^k&\sum _{1\le i_1,\cdots
,i_k\le m}\frac {\nu _{i_1}\cdots \nu_{i_k}}{e}
\frac {\theta _{i_1}\cdots \theta_{i_k}} {b}\\&\times\sum _{n=1}^{+\infty}(\sum _in_i\alpha _i^n)x^{nb^{-1}\theta _{i_1}\cdots \theta _{i_k}}.\\
\end{aligned}\end{equation}
Let $\mu \in \mathbb Q$. The coefficient of $x^{\mu }$ in
$\frac {\displaystyle xf'}{\displaystyle f}$ is
\begin{equation}\label{XXVIII}
-\sum _{k=0}^{+\infty}(-1)^k\sum _{b^{-1}\theta _{i_1}\cdots \theta
_{i_k}\mid \mu} \frac {\nu _{i_1}\cdots \nu_{i_k}}{e}\frac
{\theta _{i_1}\cdots \theta_{i_k}} {b} \sum _{i } n_i\alpha
_i^{\mu b \theta _{i_1}^{-1}\cdots  \theta _{i_k}^{-1}}.
\end{equation}

By Proposition \ref{p8}, $f\in \mathbb C[[x]]$ if and only if $\frac
{\displaystyle xf'}{\displaystyle f}\in \mathbb C[[x]]$. Therefore,
$f\in \mathbb C[[x]]$ if and only if Equation (\ref{XXVIII}) is zero for all $\mu \in \mathbb Q-\mathbb N$.

If $\mu\not\in \mathbb Q_{b}$, Equation (\ref{XXVIII}) is automatically zero .

For $\mu \in \mathbb Q_{b}-\mathbb N$, it can be uniquely
written as
\begin{equation}\label{neweq1} \mu=\lambda\cdot u,\ {\rm where}\ \lambda\in \mathbb Q_b'-\mathbb N',\ {\rm and}\ u\in \mathbb N\
{\rm s.t.}\ (u,b_0\cdots b_m)=1.
\end{equation}
Substituting (\ref{neweq1}) into (\ref{XXVIII}), we get the coefficient of $x^{\mu }$ is
\begin{equation}\label{XXIXnew}
\begin{aligned}&\sum _{k=0}^{+\infty} (-1)^k\sum _{b^{-1}\theta _{i_1}\cdots \theta
_{i_k}\mid \lambda u} \frac {\nu _{i_1}\cdots \nu_{i_k}}{e}\frac
{\theta _{i_1}\cdots \theta_{i_k}} {b} \sum _{i} n_i\alpha
_i^{\lambda ub \theta _{i_1}^{-1}\cdots  \theta _{i_k}^{-1}}\\
=&\sum _{k=0}^{+\infty} (-1)^k\sum _{b^{-1}\theta _{i_1}\cdots \theta
_{i_k}\mid \lambda} \frac {\nu _{i_1}\cdots \nu_{i_k}}{e}\frac
{\theta _{i_1}\cdots \theta_{i_k}} {b} \sum _{i} n_i\alpha
_i^{\lambda ub \theta _{i_1}^{-1}\cdots  \theta _{i_k}^{-1}}
\end{aligned},
\end{equation}
since $\lambda\in \mathbb Q_b'$ and $u\in \mathbb N,\ {\rm s.t.}\
(u,b_0\cdots b_m)=1$.

Therefore,
$f\in \mathbb C[[x]]$ if and only if \begin{equation}\label{XXIX}
\sum _{k=0}^{+\infty} (-1)^k\sum _{b^{-1}\theta _{i_1}\cdots \theta
_{i_k}\mid \lambda} \frac {\nu _{i_1}\cdots \nu_{i_k}}{e}\frac
{\theta _{i_1}\cdots \theta_{i_k}} {b} \sum _{i} n_i\alpha
_i^{\lambda ub \theta _{i_1}^{-1}\cdots  \theta _{i_k}^{-1}}=0
\end{equation}
for all $\lambda \in \mathbb Q_b'-\mathbb N'$ and
all $u\in \mathbb N,\ {\rm s.t.}\ (u,b_0\cdots b_m)=1$.

In Equation (\ref{XXIX}), substituting $u$ with $u+nb_0\cdots b_m$,
we get
\begin{equation}\label{XXX}
\begin{aligned} &\sum _{k=0}^{+\infty} (-1)^k\sum _{b^{-1}\theta _{i_1}\cdots
\theta _{i_k}\mid \lambda} \frac {\nu _{i_1}\cdots
\nu_{i_k}}{e}\frac {\theta _{i_1}\cdots \theta_{i_k}} {b}
\\&\times\sum _{i} n_i\alpha _i^{\lambda ub \theta _{i_1}^{-1}\cdots
\theta _{i_k}^{-1}}(\alpha _i^{\lambda b \theta _{i_1}^{-1}\cdots
\theta _{i_k}^{-1}b_0\cdots b_m})^n=0.
\end{aligned}
\end{equation}

In Equation (\ref{XXX}), fixed $u,\ \lambda $ and letting $n$ vary,
we obtain a family of infinite equations indexed by $n\in \mathbb N$:
\begin{equation}\label{XXXI}\begin{aligned}
\sum _{\beta }\beta ^n\big(\sum _{k=0}^{+\infty} (-1)^k&\sum
_{b^{-1}\theta _{i_1}\cdots \theta _{i_k}\mid \lambda} \frac {\nu
_{i_1}\cdots \nu_{i_k}}{e}\frac {\theta _{i_1}\cdots \theta_{i_k}}
{b} \\ &\times\sum_{\alpha _i^{\lambda b \theta _{i_1}^{-1}\cdots
\theta _{i_k}^{-1}b_0\cdots b_m}=\beta }  n_i\alpha _i^{\lambda ub
\theta _{i_1}^{-1}\cdots \theta _{i_k}^{-1}}\big)=0.
\end{aligned}
\end{equation}

In (\ref{XXXI}), since $\beta$s are distinct, the Vandermonde
determinant $$\det (\beta ^n)_{\beta ,n}\ne 0.$$ Therefore, the
coefficient of $\beta^n$ in (\ref{XXXI}) should be zero, that is,
\begin{equation}\label{XXXII}\begin{aligned}
\sum _{k=0}^{+\infty} (-1)^k&\sum _{b^{-1}\theta _{i_1}\cdots \theta
_{i_k}\mid \lambda} \frac {\nu _{i_1}\cdots \nu_{i_k}}{e}\frac
{\theta _{i_1}\cdots \theta_{i_k}} {b} \\ &\times\sum_{\alpha
_i^{\lambda b \theta _{i_1}^{-1}\cdots \theta _{i_k}^{-1}b_0\cdots
b_m}=\beta }  n_i\alpha _i^{\lambda ub \theta _{i_1}^{-1}\cdots
\theta _{i_k}^{-1}}\big)=0,
\end{aligned}
\end{equation}
for all $\lambda \in \mathbb Q_b',\ \lambda \not \in \mathbb N';$
$u\in \mathbb N,\ (u,b_0\cdots b_m)=1;$ and $\beta \in \mathbb
C^{*}.$

It is easy to see that Equation (\ref{XXXII}) implies Equation (\ref{XXIX}). Thus, the proof is complete.
\end{proof}
\begin{corollary}\label{c2}
If Equation (\ref{XII}) has a solution $f(x)\in \mathbb C[[x]]$ with $f(0)=1$, then Equation (\ref{XIII})
has a solution $g(x)\in \mathbb C[[x]]$ with $g(0)=1$, where $$H(x)=\prod _{\alpha _i\in S}(1-\alpha _ix)^{n_i}$$ is the $\mathbb N'$-cyclotomic part
of $G(x)$.
\end{corollary}
\begin{proof}
By Theorem \ref{t2}, it suffices to prove
\begin{equation}\label{XXXIII}
\begin{aligned}\sum _{k=0}^{+\infty} (-1)^k&\sum _{b^{-1}\theta _{i_1}\cdots \theta
_{i_k}\mid \lambda} \frac {\nu _{i_1}\cdots \nu_{i_k}}{e}\frac
{\theta _{i_1}\cdots \theta_{i_k}} {b} \\&\times\sum _{\substack
{\alpha _i\in S\\ \alpha _i^{\lambda b \theta _{i_1}^{-1}\cdots
\theta _{i_k}^{-1}b_0\cdots b_m}=\beta\\ }}  n_i\alpha _i^{\lambda
ub \theta _{i_1}^{-1}\cdots  \theta _{i_k}^{-1}}=0,
\end{aligned}\end{equation}
for all $\lambda \in \mathbb Q_b',\ \lambda \not \in \mathbb N';$
$u\in \mathbb N,\ (u,b_0\cdots b_m)=1;$ and $\beta \in \mathbb
C^{*}.$

Since all the elements of $S$ are $n$-th roots of unity, for some
$n\in \mathbb N'$, Equation (\ref{XXXIII}) trivially holds when
$\beta ^n\ne 1$, for all $n\in \mathbb N'$.

Otherwise, assume $\beta ^n=1$ for some $n\in \mathbb N'$. In
Equation (\ref{XXXIII}), the conditions
$$\ b^{-1}\theta _{i_1}\cdots \theta
_{i_k}\mid \lambda\ {\rm and}\ \lambda \in \mathbb Q_b'\
\Rightarrow\ \lambda b \theta _{i_1}^{-1}\cdots  \theta
_{i_k}^{-1}b_0\cdots b_m\in \mathbb N'.$$ So if
$$\alpha _i^{\lambda b \theta _{i_1}^{-1}\cdots
\theta _{i_k}^{-1}b_0\cdots b_m}=\beta,
$$
then $\alpha_i$ is an $n$-th root of unity for some $n\in \mathbb
N'$, hence $\alpha_i\in S$.

Sp we can drop the subscription $\alpha _i\in S$ in the summation of (\ref{XXXIII}):
\begin{equation}\label{XXXIIInew}
\begin{aligned}\sum _{k=0}^{+\infty} (-1)^k&\sum _{b^{-1}\theta _{i_1}\cdots \theta
_{i_k}\mid \lambda} \frac {\nu _{i_1}\cdots \nu_{i_k}}{e}\frac
{\theta _{i_1}\cdots \theta_{i_k}} {b} \\&\times\sum _{\substack
{\alpha _i^{\lambda b \theta _{i_1}^{-1}\cdots \theta
_{i_k}^{-1}b_0\cdots b_m}=\beta\\ }}  n_i\alpha _i^{\lambda ub
\theta _{i_1}^{-1}\cdots  \theta _{i_k}^{-1}}=0,
\end{aligned}\end{equation}
 Then (\ref{XXXIIInew}) holds by Theorem
\ref{t2} and the assumption $f\in\mathbb C[[x]].$
\end{proof}
Now assume \begin{equation}\label{XXXIVnew}H(x)=\prod \limits _{\alpha _i\in S}(1-\alpha
_ix)^{n_i}=\prod \limits_{d\in \mathbb N'}(1-x^d)^{m_d},\end{equation} where
$m_d=0$ for $d$ sufficiently large . Instead of Theorem \ref{t2}, we
have the following simple criterion.

For convenience, we always denote
\begin{equation}\label{XXXIV}
m_d=0 \ {\rm for\ } d\not \in \mathbb N'.
\end{equation}
\begin{theorem}\label{t3}
Let $$H(x)=\prod \limits _{d\in \mathbb N'}(1-x^d)^{m_d},\ {\rm
where}\ m_d=0\ {\rm for}\ d\gg 0.$$ Then Equation
(\ref{XIII}) has a solution $g(x)\in \mathbb C[[x]]$ with $g(0)=1$ if and only if for any $\lambda \in \mathbb Q_b'-\mathbb N'$,
the following equation holds
\begin{equation}\label{XXXV}
\sum _{k=0}^{+\infty}(-1)^k\sum _{1\le i_1,\cdots ,i_k\le m}\nu _{i_1}\cdots \nu _{i_k}m_{\lambda \theta _{i_1}^{-1}\cdots \theta _{i_k}^{-1}}=0.
\end{equation}
\end{theorem}
\begin{proof}
From the proof of Theorem \ref{t2} (see Equation (\ref{XXIX})), $g(x)\in \mathbb C [[x]]$ if and only if
the following equation
\begin{equation}\label{e25}
\sum _{k=0}^{+\infty} (-1)^k\sum _{b^{-1}\theta _{i_1}\cdots \theta
_{i_k}\mid \lambda} \frac {\nu _{i_1}\cdots \nu_{i_k}}{e}\frac
{\theta _{i_1}\cdots \theta_{i_k}} {b} \sum _{\alpha _i\in S }  n_i\alpha
_i^{\lambda ub \theta _{i_1}^{-1}\cdots  \theta _{i_k}^{-1}}=0
\end{equation} holds for any $\lambda \in \mathbb Q_b'-\mathbb N'$ and $u\in \mathbb N,\ {\rm s.t.}\ (u,\ b_0\cdots b_m)=1$.

Computing $\frac {\displaystyle xH'(x)}{\displaystyle H(x)}$ by two expressions of $H(x)$ in (\ref{XXXIVnew}), we get
\begin{equation*}
\begin{aligned}
\sum _{\alpha _i \in S}\frac {-n_i\alpha _ix}{1-\alpha _ix}&=\sum _{d\in \mathbb N'}\frac {-m_ddx^d}{1-x^d}\\
-\sum _{m=1}^{+\infty}(\sum _{\alpha _i \in  S}n_i\alpha _i ^m)x^m&=-\sum _{m=1}^{+\infty}(\sum _{\substack {d\mid m\\ d\in \mathbb N'}}dm_d)x^m.\\
\end{aligned}
\end{equation*}

So \begin{equation}\label{e26}
\sum _{\alpha _i\in S}n_i\alpha _i^m=\sum _{d\mid m, d\in \mathbb N'}dm_d=\sum _{d\mid m}dm_d
\end{equation} for any $m\in \mathbb N$. The last equality of (\ref{e26}) holds because of Equation (\ref{XXXIV}).

Substituting (\ref{e26}) into (\ref{e25}), we get
\begin{equation}\label{e27}
\begin{aligned}&\sum _{k=0}^{+\infty} (-1)^k\sum _{b^{-1}\theta _{i_1}\cdots \theta
_{i_k}\mid \lambda} \frac {\nu _{i_1}\cdots \nu_{i_k}}{e}\frac
{\theta _{i_1}\cdots \theta_{i_k}} {b} \sum _{d\mid \lambda ub\theta _{i_1}^{-1}\cdots \theta _{i_k}^{-1}}  dm_d=0.\end{aligned}
\end{equation}
We can omit $u$ from Equation (\ref{e27}):
\begin{equation}\label{e27new}
\begin{aligned}\sum _{k=0}^{+\infty} (-1)^k\sum _{b^{-1}\theta _{i_1}\cdots \theta
_{i_k}\mid \lambda} \frac {\nu _{i_1}\cdots \nu_{i_k}}{e}\frac
{\theta _{i_1}\cdots \theta_{i_k}} {b} \sum _{d\mid \lambda b\theta _{i_1}^{-1}\cdots \theta _{i_k}^{-1}}  dm_d=0
\end{aligned}
\end{equation}
since $d\in \mathbb N'$ (otherwise, $m_d=0$), $\lambda \in \mathbb Q_b'$ and
$(u,b_0\cdots b_m)=1$.

Equation (\ref{e27new}) is equivalent to
\begin{equation}\label{e28}
\sum _{k=0}^{+\infty} (-1)^k\sum _{1\le i_1,\cdots ,i_k\le m} \frac {\nu _{i_1}\cdots \nu_{i_k}}{e}\frac
{\theta _{i_1}\cdots \theta_{i_k}} {b} \sum _{d\mid \lambda b\theta _{i_1}^{-1}\cdots \theta _{i_k}^{-1}}  dm_d=0,
\end{equation} since $$d\mid \lambda b\theta
_{i_1}^{-1}\cdots \theta _{i_k}^{-1}\ {\rm and}\ d\in \mathbb N' \Rightarrow
b^{-1}\theta _{i_1}\cdots \theta
_{i_k}\mid \lambda.$$

From (\ref{XXXIV}), replacing $d$ by $db\theta _{i_1}^{-1}\cdots \theta _{i_k}^{-1}$ in Equation (\ref{e28}), we get
\begin{equation}\label{e29}
\sum _{k=0}^{+\infty} (-1)^k\sum _{1\le i_1,\cdots ,i_k\le m} \frac {\nu _{i_1}\cdots \nu_{i_k}}{e}
\sum _{d\mid \lambda}  dm_{db\theta _{i_1}^{-1}\cdots \theta _{i_k}^{-1}}=0.
\end{equation}
for any $\lambda \in \mathbb Q_b'-\mathbb N'$.

In Equation (\ref{e29}), since $$m_{db\theta _{i_1}^{-1}\cdots \theta _{i_k}^{-1}}\not =0\
\Rightarrow\
db\theta _{i_1}^{-1}\cdots \theta _{i_k}^{-1}\in \mathbb N',$$ we get $d\in \mathbb Q_b'$.
Also, the conditions: $$d\mid \lambda\ {\rm and}\ \lambda \in \mathbb Q_b'-\mathbb N'\ \Rightarrow\ d\not \in \mathbb N'.$$ Therefore, $d\in \mathbb Q_b'-\mathbb N'$, in Equation (\ref{e29}).

Changing the order of summation of (\ref{e29}), we have
\begin{equation}\label{e30}
\sum_{\substack{d\mid \lambda,\\d\in \mathbb Q_b'-\mathbb N'}}\sum _{k=0}^{+\infty} (-1)^k\sum _{1\le i_1,\cdots ,i_k\le m}\nu _{i_1}\cdots \nu_{i_k}
m_{db\theta _{i_1}^{-1}\cdots \theta _{i_k}^{-1}}=0
\end{equation}for any $\lambda \in \mathbb Q_b'-\mathbb N'$.

By the following modified version of
M\"{o}bius inversion formula (Lemma \ref{l4}), Equation (\ref{e30})
is equivalent to
\begin{equation}\label{e31}
\sum _{k=0}^{+\infty}(-1)^k\sum _{1\le i_1,\cdots ,i_k\le m}\nu _{i_1}\cdots \nu _{i_k}m_{db \theta _{i_1}^{-1}\cdots \theta _{i_k}^{-1}}=0
\end{equation}for any $d\in \mathbb Q_b'-\mathbb N'$.

Changing variable $d$ by $\lambda$ in Equation (\ref{e31}),
we get the formula (\ref{XXXV}).
\end{proof}
\begin{lemma}\label{l4} (Modified M\"{o}bius Inversion formula) Let $\{A_n\}\in \mathbb C$ be
a sequence indexed by $n\in \mathbb Q_b'-\mathbb N'$. For any $m\in \mathbb Q_b'-\mathbb N'$, define
\begin{equation}\label{e32}
B_m=\sum _{\substack {n\mid m\\ n\in \mathbb Q_b'-\mathbb N'}}A_n.
\end{equation}
We always assume (\ref{e32}) is a finite sum, i.e., there are only finitely many nonzero terms in the summation. Then
\begin{equation}\label{e33}
A_n=\sum _{\substack {m\mid n\\ m\in \mathbb Q_b'-\mathbb N'}}\mu(\frac {n}{m})B_m,
\end{equation}
where $\mu $ is the M\"{o}bius function.
\end{lemma}
\begin{proof}
\begin{equation*}
\begin{aligned}
\sum _{\substack {m|n\\ m\in \mathbb Q_b'-\mathbb N'}}\mu (\frac {n}{m})B_m&
=\sum _{\substack {m|n\\ m\in \mathbb Q_b'-\mathbb N'}}\mu (\frac {n}{m})\sum_{\substack {l|m,\\l\in \mathbb Q_b'-\mathbb N'}}A_l\\
&=\sum _{\substack {l|n\\ l\in \mathbb Q_b'-\mathbb N'}}A_l\sum _{\substack {l|m|n\\ m\in \mathbb Q_b'-\mathbb N'}}\mu (\frac {n}{m})\\
&=\sum _{\substack {l|n\\ l\in \mathbb Q_b'-\mathbb N'}}A_l\sum _{\substack {\frac {n}{m}\mid \frac {n}{l}\\ \frac {n}{m}\in \mathbb N}}\mu (\frac {n}{m})\\
&=A_n.
\end{aligned}
\end{equation*}
The second equality from the bottom is because: $$n,l\in \mathbb Q_b'-\mathbb N'\Rightarrow\frac
{\displaystyle n}{\displaystyle l} \in \mathbb N'\Rightarrow \frac
{\displaystyle m}{\displaystyle l}\in \mathbb N'\Rightarrow m\in
\mathbb Q_b'-\mathbb N'.$$
\end{proof}
\begin{theorem}\label{t4}
Let $$H(x)=\prod \limits _{d\in \mathbb N'}(1-x^d)^{m_d},\ {\rm
where}\ m_d=0\ {\rm for}\ d\gg 0.$$  Assume Equation (\ref{XIII}) has a solution
$g(x)\in \mathbb C[[x]]$ with $g(0)=1$. Then
\begin{equation}\label{e34}
g(x)=\prod _{d\in \mathbb N'}(1-x^d)^{g_d},
\end{equation}
where
\begin{equation}\label{e35}
g_d=\sum _{k=0}^{+\infty}(-1)^k\sum _{1\le i_1,\cdots ,i_k\le m}
\frac {1}{e}\nu _{i_1}\cdots \nu_{i_k}m_{bd\theta _{i_1}^{-1}\cdots
\theta _{i_k}^{-1}}.
\end{equation}
\end{theorem}
\begin{proof}
First assume that Equation (\ref{e34}) holds. Then
\begin{equation}\label{e36}
\begin{aligned}
&g(x^{b_0})^{e_0}g(x^{b_1})^{e_1}\cdots g(x^{b_m})^{e_m}\\
=&\prod _{d\in \mathbb N'}(1-x^{b_0d})^{e_0g_d}\prod _{d\in \mathbb N'}(1-x^{b_1d})^{e_1g_d}\cdots\prod _{d\in \mathbb N'}(1-x^{b_md})^{e_mg_d}\\
=&\prod _{d\in \mathbb N'}(1-x^d)^{e_0g_{d/b_0}+e_1g_{d/b_1}+\cdots +e_mg_{d/b_m}},\\
\end{aligned}
\end{equation}
where we make the convention:
\begin{equation}\label{e37}
g_d=0\ {\rm if}\ d\not \in \mathbb N'.
\end{equation}

From (\ref{e36}), Equation (\ref{XIII}) is equivalent to
\begin{equation}\label{e38}
\prod _{d\in \mathbb N'}(1-x^d)^{e_0g_d/b_0+e_1g_d/b_1+\cdots +e_mg_d/b_m}=\prod _{d\in \mathbb N'}(1-x^d)^{g_d}.
\end{equation}

By Proposition \ref{p9}, Equation (\ref{e38}) is equivalent to
\begin{equation}\label{e39}
e_0g_{d/b_0}+e_1g_{d/b_1}+\cdots +e_mg_{d/b_m}=m_d.
\end{equation}
Since for $d\in \mathbb Q'-\mathbb N',\ g_d=m_d=0$, Equation (\ref{e39}) holds for all
$d\in \mathbb Q'.$ Changing variable $d$ by $bd$ and multiplying $\frac {\displaystyle 1}{\displaystyle e}$ in
(\ref{e39}), we get
\begin{equation}\label{e40}
g_d+\nu _1g_{d\theta _1^{-1}}+\cdots +\nu _mg_{d\theta _m^{-1}}=\frac {1}{e}m_{bd},
\end{equation}
where $d\in \mathbb Q'$.

Now we will solve Equations (\ref{e37}) and (\ref{e40}) simultaneously.

Iterating (\ref{e40}), we get
\begin{equation}\label{e41}
g_d=\sum _{k=0}^{+\infty}(-1)^k\sum _{1\le i_1,\cdots ,i_k\le
m}\frac {1}{e}\nu _{i_1}\cdots \nu _{i_k}m_{bd\theta
_{i_1}^{-1}\cdots \theta _{i_k}^{-1}}
\end{equation}
for all $d\in \mathbb Q'$. Note that (\ref{e41}) is actually a finite sum.

Substituting (\ref{e41}) into (\ref{e40}), then
\begin{equation*}
\begin{aligned}
&\sum _{k=0}^{+\infty}(-1)^k\sum _{1\le i_1,\cdots ,i_k\le m}\frac {1}{e}\nu _{i_1}\cdots \nu _{i_k}m_{bd\theta _{i_1}^{-1}\cdots \theta _{i_k}^{-1}}\\
+&\sum _{k=0}^{+\infty}(-1)^k\sum _{1\le i_1,\cdots ,i_{k+1}\le m}\frac {1}{e}\nu _{i_1}\cdots \nu _{i_k}\nu _{i_{k+1}}m_{bd\theta _{i_1}^{-1}\cdots \theta _{i_{k+1}}^{-1}}\\
=&\frac {1}{e}m_{bd} .
\end{aligned}
\end{equation*}
So (\ref {e41}) is really a solution of (\ref{e40}).

Now we check the solutions (\ref {e41}) also satisfy (\ref{e37}).

Assume $d\not \in \mathbb N'$. We divide it into two cases.

\textbf{case 1:} $d\in \mathbb Q_b'-\mathbb N'$. From Theorem \ref{t3}, $g_d=0$.

\textbf{case 2:} $d\not\in \mathbb Q_b'$. Then $$m_{bd\theta _{i_1}^{-1}\cdots \theta _{i_k}^{-1}}=0\ {\rm since}\
bd\theta _{i_1}^{-1}\cdots \theta _{i_k}^{-1}=\frac {\displaystyle b^{k+1}d}{\displaystyle b_{i_1}\cdots b_{i_k}}\not \in \mathbb N.$$
By Equation (\ref{e41}), $g_d=0$, too.

Hence, $g_d=0\ {\rm if}\ d\not \in \mathbb N'$, which concludes the proof.
\end{proof}
Let $p$ be a prime. For $a\in \mathbb Z, a\ne 0$, let ${\rm
ord}_p(a)$ be the highest exponent $v$ such that $p^v$ divides $a$.
For $b\in \mathbb Z, b\ne0$, define ${\rm ord}_p(a/b) ={\rm
ord}_p(a)-{\rm ord}_p(b).$
\begin{theorem}\label{t5}
Let $$H(x)=\prod \limits _{d\in \mathbb N'}(1-x^d)^{m_d},\ {\rm
with}\ m_d=0\ {\rm for}\ d\gg 0.$$ Assume there exists a prime $p$
with ${\rm ord}_p(b_0)>{\rm ord}_p(b_i),\ {\rm for}\ 1\le i\le m.$
If Equation (\ref{XIII}) has a solution $g(x)\in \mathbb C[[x]]$
with $g(0)=1$, then
\begin{equation*}
g(x)=\prod _{d\in \mathbb N'}(1-x^d)^{g_d}
\end{equation*}
and $g_d=0$ for $d\gg 0$.
\end{theorem}

\begin{proof}
By Theorem \ref{t4}, it suffices to prove
\begin{equation}\label{e42}
g_d=\sum _{k=0}^{+\infty}(-1)^k\sum _{1\le i_1,\cdots ,i_k\le m}\frac {1}{e}\nu _{i_1}\cdots \nu _{i_k}m_{bd\theta _{i_1}^{-1}\cdots \theta _{i_k}^{-1}}
\end{equation}
is zero for $d$ sufficiently large.

For $1\le i\le m$, let
\begin{equation}\label{e43}
\theta _i=\frac {b_i}{b_0}=\rho _ip^{-w_i}\ {\rm with}\ {\rm ord}_p(\rho _i)=0\ {\rm and}\ w_i\ge 1.
\end{equation}
Obviously, $\rho _i>1$. Let $\rho =\max \{\rho _1,\cdots ,\rho _m\}>1.$

Below, we will show that, for sufficiently large $d$,
if $bd\theta _{i_1}^{-1}\cdots \theta _{i_k}^{-1}\in \mathbb N$, then it is also large. As $m_d=0$ for $d\gg 0$, this will imply $$m_{bd\theta _{i_1}^{-1}\cdots \theta _{i_k}^{-1}}=0,\ {\rm for}\ d\gg 0.$$

Now assume \begin{equation}\label{e44new}bd\theta _{i_1}^{-1}\cdots
\theta _{i_k}^{-1}=bd\rho _{i_1}^{-1}\cdots \rho
_{i_k}^{-1}p^{w_1}\cdots p^{w_k}\in \mathbb N.\end{equation}
By (\ref{e43}), the
fractional part of (\ref{e44new}), $\rho _{i_1}^{-1}\cdots \rho
_{i_k}^{-1}$, has denominator which is not divided by $p,$ so
\begin{equation}\label{e44}
bd\rho _{i_1}^{-1}\cdots \rho _{i_k}^{-1}\in \mathbb N
\end{equation}

Then we divide the proof into two cases.

\textbf{case 1:} $k\le \frac {\displaystyle 1}{\displaystyle 2}\log
_{\rho }bd$. Then
\begin{equation}\label{e45}
bd\theta _{i_1}^{-1}\cdots \theta _{i_k}^{-1}\ge bd\rho _{i_1}^{-1}\cdots \rho _{i_k}^{-1}\ge bd\rho ^{-k}\ge \sqrt {bd}.
\end{equation}

\textbf{case 2:} $k> \frac {\displaystyle 1}{\displaystyle 2}\log
_{\rho }bd$. Then, from Equation (\ref{e44}),
\begin{equation}\label{e46}
bd\theta _{i_1}^{-1}\cdots \theta _{i_k}^{-1}\ge p^{w_{i_1}}\cdots p^{w_{i_k}}\ge p^k\ge p^{ \frac {1}{2}\log _{\rho }bd}.
\end{equation}
Combining (\ref{e45}) and (\ref{e46}), we get
\begin{equation}\label{e47}
bd\theta _{i_1}^{-1}\cdots \theta _{i_k}^{-1}\ge \min \{\sqrt {bd}, p^{ \frac {1}{2}\log _{\rho }bd}\}
\end{equation}
if $bd\theta _{i_1}^{-1}\cdots \theta _{i_k}^{-1}\in \mathbb N$.

As $m_d=0$ if either $d\not \in \mathbb N$ or $d\gg 0$, from (\ref{e47}), we get $$m_{bd\theta _{i_1}^{-1}\cdots \theta _{i_k}^{-1}}=0,\ {\rm for\ all}\ 1\leq i_1,\cdots, i_k\leq m,\ {\rm if}\ d\gg 0.$$
Then, from (\ref{e42}), we get $g_d=0$ for $d\gg 0$.
\end{proof}

\section{Contradiction}
The purpose of this section is to prove the following theorem.
\begin{theorem}\label{main}
 Assume
 \begin{equation*}
H(x)=\prod \limits _{d\in \mathbb N'}\Phi
_d(x)^{c_d}\ {\rm with}\ c_d\in
\mathbb Q\ {\rm s.t.}\ c_1=-1,\ c_d=0\ {\rm for}\ d\gg 0,
\end{equation*}where $\Phi _d(x)$ be the the cyclotomic polynomial of order $d$,
defined by Equation (\ref{d1new}) below. Also let $\gcd (b_0,\cdots,b_n)=1$. Then Equation
(\ref{XIII}) has no solution $g(x)$ such that
\begin{equation*}
g(x)=\prod _{d\in \mathbb N'}(1-x^d)^{g_d}
\end{equation*}
with $g_d\in \mathbb Q$ and $g_d=0$ for $d\gg 0$.
\end{theorem}

The conclusion of Theorem \ref{main} contradicts to that of Theorem
\ref{t5} under common conditions. The proof of Theorem \ref{main} makes use of cyclotomic polynomials and
Gauss's lemma. Note that cyclotomic polynomials first appear in the
work of Cilleruelo and Ru\'{e} ~\cite{Cilleruelo}.

 We call
\begin{equation}\label{d1new}
\Phi _n(x)=\prod _{u\in (\mathbb Z /n\mathbb Z)^*}(1-\exp (u\frac
{2\pi {\rm i}}{n})x)
\end{equation}
the cyclotomic polynomial of order $n$, where $(\mathbb Z /n\mathbb
Z)^*$ denotes the set of invertible classes modulo $n$, that is,
\begin{equation*}
(\mathbb Z /n\mathbb Z)^*=\{u\in \mathbb N \mid 1\le u\le n,\
(u,n)=1\}.
\end{equation*}

Note our setting is a little different from the traditional case, in
which
\begin{equation*}
\Phi _n(x)=\prod _{u\in (\mathbb Z /n\mathbb Z)^*}(x-\exp (u\frac
{2\pi {\rm i}}{n})).
\end{equation*}However, they differ up to multiplying by $\pm 1$. The
remarkable point is that $\Phi _n(0)=1$ in our setting.

The following facts about cyclotomic polynomials are well known.

(1) $\Phi _n(x)$ is irreducible in $\mathbb Z[x]$. As a consequence,
if a polynomial $P(x)\in \mathbb Z[x]$ vanishes at a primitive root
of unity of order $n$, then there exists a positive integer $s$ such
that $P(x)=\Phi _n(x)^sQ(x)$, where $Q(x)\in \mathbb Z[x]$ and
$Q(\xi )\ne 0$ for all $\xi $,  $n$-th primitive roots of unity.

(2) $\{\Phi _n(x)\mid n\in \mathbb N\}$ and $\{1-x^n\mid n\in
\mathbb N\}$ can represent each other:
\begin{equation}\label{ee1}
1-x^n=\prod _{d|n}\Phi _d(x),\ \Phi _n(x)=\prod _{d|n}(1-x^d)^{\mu
(\frac {n}{d})}
\end{equation}
where $\mu (\cdot )$ is the M\"{o}bius function. This also implies that $${\rm deg}\ \Phi _n(x)=n\prod _{p|n}(1-\frac
{\displaystyle 1}{\displaystyle p})=\varphi (n),$$ where $\varphi$ is the Euler
function.

To Theorem \ref{main}, we need the following lemma.
\begin{lemma}\label{ll1}
\begin{equation}\label{ee2}
 \Phi _d(x^a)=\prod _{d\langle a|d)|f|ad} \Phi _f(x),
 \end{equation}
 where, for $a,d\in \mathbb N,$
 \begin{equation}\label{ee3}
 \langle a|d)=\prod _{p|(a,d)}p^{{\rm ord}_p(a)}.
 \end{equation}
\end{lemma}
\begin{proof}
\begin{equation}\label{ee4}
\begin{aligned}
\Phi _d(x^a)&=\prod _{u\in (\mathbb Z /d\mathbb Z)^*}(1-\exp (u\frac
{2\pi {\rm i}}{d})x^a)\\
&=\prod _{u\in (\mathbb Z /d\mathbb Z)^*}\prod _{1\le k\le a}(1-\exp
(u\frac {2\pi {\rm i}}{ad})\exp (k\frac {2\pi {\rm i}}{a})x)\\
&=\prod _{u\in (\mathbb Z /d\mathbb Z)^*}\prod _{1\le k\le a}(1-\exp (\frac {u+kd}{ad}2\pi {\rm i})x)\\
\end{aligned}
\end{equation}
Assume $$\xi =\exp (\frac {u+kd}{ad}2\pi {\rm i})$$ is a primitive
$f$-th root of unity. Then $f$ is the smallest positive integer such
that $\xi ^f=1$, i.e., $ad|f(u+kd)$.

Obviously, $f\mid ad$. Since $(u,d)=1$, we have $(u+kd,d)=1$. Then Equation
(\ref{ee3}) implies that $(u+kd,\langle a|d)d)=1$. Since $\langle a|d)d\mid
f(u+kd)$, we get $\langle a|d)d\mid f$. Thus $\langle a|d)d\mid
f\mid ad$.

Since each factor $$(1-\exp (\frac {\displaystyle
u+kd}{\displaystyle ad}2\pi {\rm i})x)$$ in Equation (\ref{ee4}) appears one
time, we get
\begin{equation}\label{ee5}
\Phi _d(x^a)\mid \prod _{d\langle a|d)\mid f\mid ad}\Phi _f(x).
\end{equation}

The degree of the left hand side of (\ref{ee5}) is $a\varphi (d)$.
The degree of the right hand side is
\begin{equation*}
\begin{aligned}
\sum _{d\langle a|d)\mid f\mid ad}\varphi (f)&=\sum _{f'\mid \frac
{a}{\langle a|d)}}\varphi (f'd\langle a|d))\\&=\varphi(d\langle
a|d))\sum _{f'\mid \frac
{a}{\langle a|d)}}\varphi (f')\ \ \ \ {\rm as}\ (\frac
{a}{\langle a|d)}, d\langle
a|d))=1\\
&=\varphi (d\langle a|d))\frac{a}{\langle a|d)}
\\&=\varphi (d)a\ \ \ \ {\rm as}\  p\mid\langle
a|d)\Rightarrow p\mid d.
\end{aligned}
\end{equation*}
Therefore,
$${\rm deg}\ \Phi _d (x^a) ={\rm deg}
\prod \limits_{d\langle a|d)|f|ad}\Phi _f(x).$$ Since their constant
terms both equal to 1, they must be equal.
\end{proof}
Now assume
\begin{equation}\label{ee6}
g(x)=\prod _{d\in \mathbb N}\Phi _d(x)^{h_d},\ {\rm with}\ h_d=0\
{\rm for}\ d\gg0.
\end{equation}

Let $a\in \mathbb N$, by Lemma \ref {ll1},
\begin{equation}\label{ee7}
\begin{aligned}
g(x^a)&=\prod _{d\in \mathbb N}\Phi _d(x^a)^{h_d}\\
&=\prod _{d\in \mathbb N}\prod _{\langle a|d)d\mid f\mid ad}\Phi
_f(x)^{h_d}\\
&=\prod _{f\in \mathbb N}\Phi _f(x)^{\sum_{\langle a|d)d\mid f\mid
ad}h_d}.\\
\end{aligned}
\end{equation}

For fixed $f$,
\begin{equation}\label{ee8new}\langle a|d)d\mid f\mid ad\Leftrightarrow \langle a|d)a^{-1}d\mid fa^{-1}\mid d. \end{equation}
Since $a/\langle a|d)$ and $d$ have no common divisor, we have
\begin{equation}\label{ee8}
{\rm ord} _p(fa^{-1})\left \{\begin{aligned}\ &={\rm ord}_p(d),
&{\rm
 if}\ p\mid d,\\
 &\le 0,\ &{\rm if}\ p\nmid d,\\
 \end{aligned}
 \right.
 \end{equation}
 where $p$ is any prime.

 For positive $y\in \mathbb Q$, denote
 \begin{equation}\label{ee9}
 [y]=\prod _{{\rm ord}_p(y)>0}p^{{\rm ord}_p(y)}.
 \end{equation}

 From (\ref{ee8new}) and (\ref{ee8}), we have, for fixed $f$,
 \begin{equation}\label{ee10}
 \langle a|d)d\mid f\mid ad\Rightarrow d=[\frac {f}{a}].
 \end{equation}

 Combining (\ref{ee7}) and (\ref{ee10}), we get
 \begin{equation}\label{ee11}
 g(x^a)=\prod _{f\in \mathbb N}\Phi _f(x)^{h_{[f/a]}}.
 \end{equation}

From Equation (\ref{ee11}), we get the following formula.
\begin{lemma}\label{ll2}
Assume \begin{equation*}\label{ee6}
g(x)=\prod _{d\in \mathbb N}\Phi _d(x)^{h_d},\ {\rm with}\ h_d=0\
{\rm for}\ d\gg0.
\end{equation*}
 Then
\begin{equation}\label{ee12}
g(x^{b_0})^{e_0}g(x^{b_1})^{e_1}\cdots g(x^{b_m})^{e_m}=\prod _{d\in
\mathbb N}\Phi _d(x)^{\sum _{i}e_ih_{[d/b_i]}}.
\end{equation}
\end{lemma}

To prove the main result of this section, we also need Gauss's
Lemma. Now we recall it.

Let $p(x)=a_0+a_1x+\cdots +a_nx^n$ be a non-zero polynomial in
$\mathbb Z[x]$. If the greatest common divisor of $a_0,a_1,\cdots ,a_n$ is 1,
then $p(x)$ is called a primitive polynomial.

Every non-zero polynomial $q(x)\in \mathbb Q [x]$ can be uniquely
written as
\begin{equation*}
q(x)=cq_1(x)
\end{equation*}
with $c>0$ and $q_1(x)\in \mathbb Z [x]$ being primitive. We call
$c$ the content of $q(x)$ and denote it by ${\rm cont} (q)$. The
following version of Gauss's Lemma will be found in page 181 of Lang
\cite{Lang}, Theorem 2.1 of Chapter IV.
\begin{theorem}\label{Gauss} (Gauss's Lemma) Let $p,q\in \mathbb Q[x]$ be
non-zero polynomials. Then
\begin{equation*}
{\rm cont}(p\cdot q)={\rm cont}(p)\cdot {\rm cont}(q).
\end{equation*}
\end{theorem}

Finally, we can prove Theorem \ref{main}.

\begin{proof}
Assume Equation (\ref{XIII}) has a solution $g(x)$ such that
 \begin{equation}\label{ee6new}
g(x)=\prod _{d\in \mathbb N'}(1-x^d)^{g_d},\ {\rm with}\ g_d\in
\mathbb Q\ {\rm and}\ g_d=0\
{\rm for}\ d\gg0.
\end{equation} From (\ref{ee1}), we have
\begin{equation*}
g(x)=\prod _{d\in \mathbb N'}\Phi _d(x)^{h_d}\ {\rm with}\ h_d\in
\mathbb Q\ {\rm and}\ h_d=0\ {\rm for}\ d\gg 0.
\end{equation*}

From Lemma \ref{ll2}, we get
\begin{equation}\label{ee13}
\prod _{d\in \mathbb N'}\Phi _d(x)^{\sum
_{i=0}^{m}e_ih_{[d/b_i]}}=\prod _{d\in \mathbb N'}\Phi _d(x)^{c_d}.
\end{equation}

Since $\Phi _d(x)$ is irreducible in $\mathbb Z[x]$, taking some power of Equation (\ref{ee13}) if necessary,  we get the following equations by the uniqueness
factorization property of $\mathbb Z[x]$:
\begin{equation}\label{ee14}
\sum _{i=0}^{m}e_ih_{[d/b_i]}=c_d\ {\rm for\ all}\ d\in \mathbb N'.
\end{equation}

Since $\gcd(b_0,\cdots ,b_m)=1$, there exists a prime $p$ such that
$p\mid b_0$ but $p\nmid b_i$ for some $1\le i\le m$. Taking $d=p^n\
(n\ge 0)$ in (\ref{ee14}), we get the following equations
\begin{equation}\label{ee15}
a_0h_{[p^{n}]}+a_1h_{[p^{n-1}]}+\cdots +a_th_{[p^{n-t}]}=c_{p^n}\
(n\ge 0)
\end{equation}
where
\begin{equation}\label{15++} t\ge 1\  {\rm and}\  a_0a_t\ne 0.
\end{equation}

To simplify the notations, let $h_{p^n}=H_n'$ and $c_{p^n}=C_n'$.
Then Equation (\ref{ee15}) can be written explicitly as
\begin{equation}\label{ee16}
\left \{ \begin{aligned} &a_0H_0'&+&a_1H_0'&+\cdots &+a_tH_0'&=&C_0'=-1\\
&a_0H_1'&+&a_1H_0'&+\cdots &+a_tH_0'&=&C_1'&\\
&a_0H_2'&+&a_1H_1'&+\cdots &+a_tH_0'&=&C_2'&\\
&&\cdots \ &\cdots \\
&a_0H_t'&+&a_1H_{t-1}'&+\cdots &+a_tH_0'&=&C_t'&\\
&&\cdots \ &\cdots  \\
&a_0H_{l+1}'&+&a_1H_{l}'&+\cdots &+a_tH_{l+1-t}'&=&C_{l+1}'.&\\
&&\cdots \ &\cdots  \\
\end{aligned}
\right.
\end{equation}So
\begin{equation}\label{ee17}
H_0'=-\frac{1}{A},\ {\rm where}\ A=\sum \limits_{i=0}^{t}a_i=\sum \limits_{i=0}^{m}e_i.
\end{equation}

Substracting the other equations of (\ref{ee16}) by the first
equation, and letting $H_{i-1}=H_i'-H_0',\ C_{i-1}=C_i'-C_0'\ (i\ge
1)$, we get
\begin{equation}\label{ee18}
\left \{ \begin{aligned} &a_0H_0&=C_0\\
&a_0H_1+a_1H_0&=C_1\\
&\cdots \ \cdots &\\
&a_0H_t+a_1H_{t-1}+\cdots +a_tH_0&=C_t\\
&\cdots \ \cdots &\\
&a_0H_l+a_1H_{l-1}+\cdots +a_tH_{l-t}&=C_l.
\\&\cdots \ \cdots  &\\
\end{aligned} \right.
\end{equation}

Note that \begin{equation}\label{ee19} H_k=-H_0'=\frac {1}{A},\ C_l=-C_0'=1
\end{equation}
for $k,\ l$ sufficiently large.

The Equation (\ref{ee18}) is
equivalent to the following identity in $\mathbb Q[[z]]$.
\begin{equation}\label{ee20}\begin{aligned}
&(a_0+a_1z+\cdots +a_tz^t)(H_0+H_1z+\cdots
+H_kz^k+\cdots)\\=&C_0+C_1z+C_2z^2+\cdots+C_lz^l+\cdots
\end{aligned}\end{equation}

Let $r$ (resp. $s$) be the largest $k$ (resp. $l$) such that
(\ref{ee19}) does not hold. Substituting (\ref{ee19}) into (\ref{ee20}), we get
\begin{equation}\label{ee21}
\begin{aligned}
&(a_0+a_1z+\cdots +a_tz^t)(H_0+H_1z+\cdots +H_rz^r+\frac {1}{A}\frac {z^{r+1}}{1-z})\\
=&C_0+C_1z+\cdots +C_sz^s+\frac {z^{s+1}}{1-z}
\end{aligned}
\end{equation}

Multiplying both sides of Equation (\ref{ee21}) by $1-z$, we get
\begin{equation}\label{ee22}
\begin{aligned}&\ \ \ (a_0+\cdots +a_tz^t)((H_0+\cdots
+H_rz^r)(1-z)+\frac {z^{r+1}}{A})\\
=&C_0+(C_1-C_0)z+\cdots +(C_s-C_{s-1})z^n+(1-C_s)z^{s+1}.
\end{aligned}
\end{equation}

The right hand side of (\ref{ee22}) is a primitive polynomial, since
their coefficients sum to 1. Let $d=\gcd(a_0,a_1,\cdots ,a_t)$.
From (\ref{ee22}) and Gauss's
Lemma, the content of
\begin{equation}\label{neweq2}
(H_0+\cdots
+H_rz^r)(1-z)+\frac {z^{r+1}}{A}=\frac{1}{d}.
\end{equation} So the following polynomial
\begin{equation}\label{neweq3}
d(H_0+\cdots
+H_rz^r)(1-z)+\frac {dz^{r+1}}{A}\in
\mathbb Z[z],
\end{equation} is primitive.

Evaluating (\ref{neweq3}) at $z=1$, we get
\begin{equation}\label{ee23}
\frac {d}{A}=\frac {d}{\sum _{i=0}^{t}a_i}\in \mathbb Z.
\end{equation}
From (\ref{15++}), $$\sum _{i=0}^{t}a_i>a_0\ge d>0.$$ A contradiction!
\end{proof}

\section{Proof of Theorem \ref{II}}
\vskip 0.3cm
\begin{proof}
Since $P(x)\in \mathbb Z[x]$, with $P(0)=1$ and $P(1)\ne 0$, it can be factored uniquely as
\begin{equation}
P(x)=\prod _{d\in \mathbb N'-\{1\}}\Phi _d(x)^{c_d}R(x)=Q(x)R(x),
\end{equation}
with $R(x)\in \mathbb Z[x]$ and $R(\xi)\ne 0$ if $\xi^n=1$, for some $n\in \mathbb N'$.

Let
\begin{equation*}
G(x)=\frac {P(x)}{1-x}.
\end{equation*}
Then $H(x)$, the $\mathbb N'$-th cyclotomic part of $G(x)$, can be written as
\begin{equation}\label{eee1}
H(x)=\frac {Q(x)}{1-x}=\prod _{d\in \mathbb N'}\Phi _d(x)^{c_d}=\prod _{d\in \mathbb N'}(1-x^d)^{m_d}
\end{equation}
with $c_d,\ m_d\in \mathbb Z$ such that $c_1=-1$ and $c_d=m_d=0$ for $d\gg 0$.

Assume the equation
\begin{equation}\label{eee2}
f(x^{b_0})^{e_0}f(x^{b_1})^{e_1}\cdots f(x^{b_m})^{e_m}=G(x)
\end{equation}
 has a solution $f(x)\in \mathbb C[[x]]$ with $f(0)=1.$

By Corollary \ref{c2}, the equation
\begin{equation}\label{eee3}
g(x^{b_0})^{e_0}g(x^{b_1})^{e_1}\cdots g(x^{b_m})^{e_m}=H(x)
\end{equation}
has a solution $g(x)\in \mathbb C[[x]]$ with $g(0)=1$.

From Theorem \ref{t4}, we get
\begin{equation}\label{eee4}
g(x)=\prod _{d\in \mathbb N'}(1-x^d)^{g_d}
\end{equation}
with $g_d\in\mathbb Q$.

By the assumption and Theorem \ref{t5}, we have
\begin{equation}\label{eee5}
g_d=0\ {\rm for}\ d\gg 0.
\end{equation}

If $\gcd (b_0,b_1,\cdots ,b_m)=1$, from Equation (\ref{eee1}) and Theorem \ref{main}, we know (\ref{eee5}) is impossible.
A contradiction!

Let $\gcd(b_0,b_1,\cdots ,b_m)=d>1$. Then the left hand side of Equation (\ref{eee2}) is a power series with indeterminate $x^d$, but the coefficient of
$x^n$ of the right hand side is a nonzero constant for $n$ large enough. So we still get a contradiction.

Summing up, Equation (\ref{eee2}) has no solution $f(x)\in \mathbb C[[x]]$ with $f(0)=1$, which concludes the proof.
\end{proof}

\section{Conjectures and Remarks}
In this section, we will give a conjectural answer to the question of S\'{a}rkozy and S\'{o}s in the case that all the coefficients of
linear forms are positive.

Let $k$ be an integer greater than 1. For $m\ge 1$, let
\begin{equation*}
{\rm M}=\{(1,1),(k,1),\cdots ,(k^{m-1},1),(k^m,1)\}.
\end{equation*}

Motivated by Ruzsa's example, let

\begin{equation*}
\mathcal A=\{\sum _{i=0}^{+\infty}\varepsilon _ik^{(m+1)i},\ \varepsilon _i\in \{0,1,\cdots ,k-1\}\}.
\end{equation*}
By the uniqueness of $k$-adic representation, we get the representation function
\begin{equation*}
r_{\rm M}(n,\mathcal A)=\#\{(a_0,a_1,\cdots ,a_m)\mid a_0+ka_1+\cdots +k^ma_m=n,\ a_i\in \mathcal A\}
\end{equation*}
is 1 for all $n\ge 0$.

We conjecture that these are the complete answers to S\'{a}rkozy and S\'{o}s's question.\\
{\bf Conjecture.}
For $m\ge 1$, let
\begin{equation*}
{\rm M}=\{(b_0,e_0),(b_1,e_1),\cdots ,(b_m,e_m)\}
\end{equation*}
with $1\le b_0 < b_1<\cdots <b_m$. There exists an infinite subset $\mathcal A$ such that $r_{\rm M}(n, \mathcal A)$ is constant for $n$
large enough only if
\begin{equation*}
{\rm M}=\{(1,1),(k,1),\cdots ,(k^{m-1},1),(k^m,1)\}
\end{equation*}
for some $k>1$.

Our initial plan is to prove the above conjecture for the case $b_0\ge 2$. But it is not successful.
The only problem happens in Theorem \ref{t5}. If it can be improved, the case of $b_0=2$ is done by our rest arguments.

Note that Theorem \ref{t5} does not hold for general $b_0,b_1,\cdots ,b_m$ with $b_0\ge 2$ and general $H(x)$.
For example, take $\mathcal A$ to be the set in Ruzsa's example, that is,
\begin{equation*}
\mathcal A=\{\sum _{i=0}^{+\infty}\varepsilon _i2^{2i},\ \varepsilon _i\in \{0,1\}\}.
\end{equation*}
Let $g(x)$ be the generating function of $\mathcal A$.
From Page 4, we have
\begin{equation*}
g(x)g(x^2)=\frac {1}{1-x}\ \ {\rm and}\ \ g(x)=\prod _{n=1}^{+\infty}(\frac {1}{1-x^{2^n}})^{(-1)^n}.
\end{equation*}
Take $(b_0,b_1,b_2,b_3)=(2,3,4,6)$, we get
\begin{equation*}
g(x^2)g(x^3)g(x^4)g(x^6)=\frac {1}{1-x^2}\frac {1}{1-x^3}.
\end{equation*}
This shows that Theorem \ref{t5} is not true in general.

For the case $b_0=1$, the equation
\begin{equation*}
f(x^{b_0})^{e_0}f(x^{b_1})^{e_1}\cdots f(x^{b_m})^{e_m}=\frac {P(x)}{1-x}=G(x)
\end{equation*}
always has a power series solution
\begin{equation*}
f(x)=\prod _{k=0}^{+\infty }\left (\prod _{1\le i_1,\cdots ,i_k\le m}G(x^{b _{i_1}\cdots ,b _{i_k}})^{e^{-1}\nu _{i_1}\cdots \nu _{i_k}}\right )^{(-1)^k}.
\end{equation*}
To solve S\'{a}kozy and S\'{o}s's question, we need to decide whether all the coefficients of $f(x)$ belong to $\{0,1\}$.

It seems difficult to treat the coefficients of infinite products.
For example, the Ramanujan tau function $\tau:\ \mathbb
N\rightarrow\mathbb Z$ is defined by the following identity in $\mathbb C[[q]]$:
\begin{equation*}
q\prod _{n=1}^{+\infty}(1-q^n)^{24}=\sum _{n=1}^{+\infty}\tau (n)q^n.
\end{equation*}
Lehmer conjectured that $\tau(n)\neq0$ for all $n$, an assertion
sometimes known as Lehmer's conjecture. Lehmer verified the
conjecture for $n<214928639999$ (See page 22 of \cite{Apostol}).
This conjecture is still open now.

The above arguments suggest that the case of $b_0=1$ is more difficult.

\textbf{Acknowledgement:}  This work is partially supported by
National Natural Science Foundation of China (Grant No. 11101424
 and
 Grant No. 11071277).


\vskip 2cm
\begin{thebibliography}{??}
\parskip=0pt
\itemsep=0pt
\bibitem{Apostol} T. M. Apostol, Modular Functions and Dirichlet Series in Number Theory, Springer-Verlag, New York, 2nd ed, (1997).
\bibitem{Cilleruelo} J. Cilleruelo, J. Ru\'{e}, On a Question of S\'{a}rkozy and S\'{o}s for Bilinear forms, Bulletin of the
London Mathematical Society 4, 2(2009), 274-280.
\bibitem{Dirac}G. A. Dirac, Note on a Problem in Additive Number Theory, J. London Math. Soc. 26 (1951) pp.
312-313.
\bibitem{Erdos} P. Erd\"{o}s, P. Tur\'{a}n, On a problem of Sidon in Additive Number Theory, and on some related problems,
J. London Math. Soc. 16 (1941) pp. 212-215.
\bibitem{Lang} S. Lang, Algebra, Revised third edition. Graduate Texts in Mathematics, 211.
Springer-Verlag, New York, 2002.
\bibitem{Moser} L. Moser, An Application of Generating Series, Mathematics Magazine (1) 35 (1962) 37-38.
\bibitem{Sarkozy}A. S\'{a}rkozy, V. T. S\'{o}s, On additive representation functions, The mathematics of Paul Erd\"{o}s I (eds
P. Erd\"{o}s, R. L. Graham and J. Nesetril), Algorithms Combin. 13 (Springer, Berlin, 1997), pp 129-150
\bibitem{Rue}J. Ru\'{e}, On Polynomial Representation Function for Multilinear Forms,   arXiv:math/1104.2716v1.
312-313.
\bibitem{Stanley}R. P. Stanley, Enumerative Combinatorics, vol. II, Cambridge University Press, 1999.
\bibitem{Stein} E. Stein, R. Shakarchi, Complex Analysis, Princeton Lectures in Analysis II, Princeton
University Press, 2003.
\end{thebibliography}
\end{document}